\documentclass[11pt,twoside,leqno]{article}
\usepackage{amsmath,amsthm,amssymb}
\numberwithin{equation}{section}
\usepackage[dvips]{graphicx,psfrag}
\newcommand{\mG}{\mathbf{\G}}

\newcommand{\Det}{\mathrm{Det}}

\newcommand{\che}[1]{\check{#1}}

\newcommand{\tz}{\widetilde{\zeta}}

\makeatletter

\ifcase\@ptsize
\def\rpkern{\mathchoice{\kern-1.45em}{\kern-1.11em}{}{}}%
\def\grpkern{\mathchoice{\kern-1.013em}{\kern-0.825em}{}{}}%
\or
\def\rpkern{\mathchoice{\kern-1.44em}{\kern-1.11em}{}{}}%
\def\grpkern{\mathchoice{\kern-1.00em}{\kern-0.81em}{}{}}%
\or
\def\rpkern{\mathchoice{\kern-1.472em}{\kern-1.14em}{}{}}%
\def\grpkern{\mathchoice{\kern-1.00em}{\kern-0.815em}{}{}}%
\fi

\def\minibullet{\mathchoice%
{\raise0.2ex\hbox{$\scriptstyle\bullet$}}%
{\raise0.26ex\hbox{$\scriptscriptstyle\bullet$}}{}{}}
\def\butabullet{\mathchoice%
{\raise0.8ex\hbox{$\scriptstyle\bullet$}{\kern-0.365em}%
\lower0.4ex\hbox{$\scriptstyle\bullet$}}%
{\raise0.75ex\hbox{$\scriptscriptstyle\bullet$}{\kern-0.335em}%
\lower0.25ex\hbox{$\scriptscriptstyle\bullet$}}{}{}}

\def\customprod#1#2%
{\operatorname{\coprod\kern#2{#1}\kern#2\prod}\displaylimits}

\newcommand{\bC}{\mathbb{C}}

\newcommand{\bN}{\mathbb{N}}

\newcommand{\bQ}{\mathbb{Q}}
\newcommand{\bR}{\mathbb{R}}

\newcommand{\bZ}{\mathbb{Z}}

\newcommand{\cD}{\mathcal{D}}

\newcommand{\cY}{\mathcal{Y}}

\renewcommand{\a}{\alpha}
\renewcommand{\b}{\beta}
\newcommand{\g}{\gamma}
\renewcommand{\d}{\delta}

\newcommand{\z}{\zeta}

\renewcommand{\l}{\lambda}

\renewcommand{\t}{\tau}

\newcommand{\G}{\Gamma}

\renewcommand{\Re}{\mathrm{Re}\,}

\renewcommand{\det}{\mathrm{det}\,}

\newcommand{\Gauss}[1]{\lfloor{#1}\rfloor}

\newcommand{\p}{\partial}
\newcommand{\DS}[1]{\displaystyle{#1}}

\newcommand{\wt}[1]{\widetilde{#1}}

\newcommand{\boldtitle}[1]{\title{\bfseries #1}}

\newenvironment{MSC}{%
\smallbreak
\noindent \textbf{2010\ Mathematics Subject Classification\,:}}

\newenvironment{keywords}{%
\noindent\textbf{Key words and phrases\,:}\itshape}

\theoremstyle{theorem}
\newtheorem*{multitheorem}{\variable@name}

\theoremstyle{definition}
\newcommand{\variable@name}{Theorem}
\newtheorem*{multiproclaim}{\variable@name}

\theoremstyle{plain}
\newtheorem{thm}{Theorem}[section]
\newtheorem{prop}[thm]{Proposition}
\newtheorem{lem}[thm]{Lemma}
\newtheorem{cor}[thm]{Corollary}

\theoremstyle{definition}

\newtheorem{example}[thm]{Example}
\newtheorem{remark}[thm]{Remark}

\boldtitle{Factorization formulas for higher depth determinants\\
 of the Laplacian on the $n$-sphere}
\author{Yoshinori
 Yamasaki\thanks{Partially supported by Grant-in-Aid for Young Scientists (B) No. 21740019.}} 
\date{\today}

\begin{document}
%===========================================================================
%===========================================================================

\setlength{\baselineskip}{15pt}
\maketitle

\begin{abstract}
 We explicitly give factorization formulas for higher depth determinants,
 which are defined via derivatives of the spectral zeta function at non-positive integer points,
 of Laplacians on the $n$-sphere
 in terms of the multiple gamma functions.  
\begin{MSC}
 {\it Primary} 11M36.
% Selberg zeta functions and regularized determinants;
% applications to spectral theory, Dirichlet series, Eisenstein series, etc.
% Explicit formulas
\end{MSC} 
\end{abstract}

%================================================================
\section{Introduction}
%================================================================

 Let $A$ be an operator having only real discrete spectra.
 Let $S(A)$ be the set of all spectra
 and $S^{+}(A)$ the set of all positive spectra of $A$ including multiplicity.
 The aim of the present paper is, for a positive integer $r$,
 to study the ``higher depth determinants'' of $A$ defined by 
\[
 \Det_{r}(A):
=\exp\Bigl(-\frac{\p}{\p w}\z_{A}(w)\Bigl|_{w=1-r}\Bigl).
\]
 Here,
 $\z_{A}(w)$ is the spectral zeta function of $A$ defined by
\begin{equation}
\label{def:speczeta}
 \z_{A}(w):=\sum_{\l\in S^{+}(A)}\frac{1}{\l^w}.
\end{equation}
 To guarantee the well-definedness,
 we assume that $\z_{A}(w)$ converges absolutely in some right half plane,  
 can be continued meromorphically to the region containing $w=1-r$ and 
 is, in particular, holomorphic at the point.
 The higher depth determinant clearly gives a generalization
 of the usual (zeta-) regularized determinant $\det(A)$ of $A$, $\det(A):=\Det_1(A)$.
 It is first introduced and studied in \cite{{KurokawaWakayamaYamasaki}}
 for the case where $A$ is the Laplacian
 on a compact Riemann surface of genus greater than one with negative constant curvature.
 In that paper,
 a generalization of the determinant expression of the Selberg zeta function
 (\cite{DHP,Sa,V}) is obtained.
 See \cite{WakayamaYamasaki} for their arithmetic analogues.
 The origin of such higher depth functions goes back to Milnor's gamma functions;
 in \cite{M}, 
 to construct a function satisfying the Kubert identity,
 which plays an important role in the study of Iwasawa theory, 
 Milnor studied the following ``higher depth gamma function'':
\[
 \mG_r(z):=\exp\Bigl(\frac{\p}{\p w}\z(w,z)\Bigl|_{w=1-r}\Bigl),
\] 
 where $\z(w,z):=\sum^{\infty}_{m=0}(m+z)^{-w}$ is the Hurwitz zeta function.
 Notice that, from the Lerch formula % (\cite{Lerch1894})
 $\exp(\frac{\p}{\p w}\z(w,z)\bigl|_{w=0})=\frac{\G(z)}{\sqrt{2\pi}}$
 where $\G(z)$ is the classical gamma function,
 we have $\mG_{1}(z)=\frac{\G(z)}{\sqrt{2\pi}}$, 
 and hence the Milnor gamma function in fact gives a generalization of the gamma function.
 From this, one can easily see that 
 the higher depth determinants are analogues of the Milnor gamma functions.

 In the present paper,
 we study the higher depth determinants 
 of Laplacians on the unit $n$-sphere
 $S^{n}:=\{(x_1,\ldots,x_{n+1})\in\bR^{n+1}\,|\,x_1^2+\cdots+x_{n+1}^2=1\}$
 with the standard metric.
 Let $\Delta_n$ be the Laplacian on $S^{n}$.
 It is well known that $\z_{\Delta_n}(w)$ converges absolutely for $\Re(w)>\frac{n}{2}$
 and can be continued meromorphically to the whole complex plane $\bC$
 with possible simple poles at
 $w=\frac{n}{2},\frac{n}{2}-1,\ldots,\frac{n}{2}-\Gauss{\frac{n-1}{2}}$,
 where $\Gauss{x}$ is the largest integer not exceeding $x$,
 and is holomorphic at non-positive integer points.
 (These properties have been verified more generally.
 Namely, let $\Delta_M$ be the Laplacian on a compact Riemannian manifold $M$ of dimension $n$.
 Then, the spectral zeta function $\z_{\Delta_M}(w)$ satisfies the properties stated above.
 See, e.g., \cite{{Rosenberg1997}}.)  
 In particular, we can define the higher depth determinant $\Det_r(\Delta_n)$.

 Before stating our results, let us recall the case $r=1$.
 In \cite{Vardi1988} (see also \cite{Kumagai1999}),
 it is shown that the determinant $\det(\Delta_n)$ can be expressed as
\begin{align}
\label{for:vardi}
 \det(\Delta_n)
&=A_n^{B_n}e^{C_n}\prod^{n}_{l=1}\G_l\bigl(\frac{1}{2}\bigr)^{Q_{n,l}}
=\a_n^{\b_n}e^{\g_n}\prod^{n-1}_{k=0}e^{\t_{n,k}\z'(-k)},
\end{align}
 where $\G_l(z)$ is the Barnes multiple gamma function (\cite{Barnes}), 
 $\z(w):=\sum^{\infty}_{m=1}m^{-w}$ is the Riemann zeta function, and 
 $A_n$, $B_n$, $C_n$, $Q_{n,l}$, $\a_n$, $\b_n$, $\g_n$ and $\t_{n,k}$
 are some computable rational numbers.
 Notice that the second equation immediately follows from the fact that 
 $\G_l(\frac{1}{2})$ can be expressed in terms of $\z'(-k)$ for $k=0,1,\ldots,l-1$
 (see \cite[Theorem~1.3]{Vardi1988} or, more explicitly, \eqref{for:G1/2}).

 The Laplacian treated in this paper is, as a generalization of $\Delta_n$, of the following form: 
\[
 L_n(s):
=\Delta_n+\check{n}^2-s^2
=\Delta_n+\frac{n-1}{4n}R_n-s^2,
\]
 where $\check{n}:=\frac{n-1}{2}$, $s\in\bR$ and $R_n=n(n-1)$ is the scalar curvature of $S^n$.
 This form of the Laplacian is also studied in \cite{Dowker1994,Moller2009}.
 Notice that
 $\Delta_n=L_n(\check{n})$ and 
 $Y_n=L_n(\frac{1}{2})$, where $Y_n:=\Delta_n+\frac{n(n-2)}{4}$
 is the conformal Laplacian (or the Yamabe operator) on $S^n$.
 The following is our main result,
 which gives factorization formulas for the higher depth determinants of $L_n(s)$.
 
\begin{thm}
\label{thm:main}
 Let $s\in I_n:=(-\check{n}-1,\check{n}+1)$.
 Put $\xi^{\pm}_n(s):=1+\check{n}\pm s$.
 Then, there exists some even polynomials $f_{n,r}(s)$ and $\b_{n,r}(s,l)$
 with rational number coefficients such that
\begin{equation}
\label{for:DetrB-intro}
 \Det_{r}\bigl(L_n(s)\bigr)
=\bigl(\check{n}^2-s^2\bigr)^{(\check{n}^2-s^2)^{r-1}\d_{|s|<\check{n}}}\cdot e^{f_{n,r}(s)}
\prod^{2r+n-2}_{l=1}
\Bigl(\G_l\bigl(\xi^{+}_n(s)\bigr)\G_l\bigl(\xi^{-}_n(s)\bigr)\Bigr)^{\b_{n,r}(s,l)},
\end{equation}
 where $\d_{|s|<\check{n}}:=1$ if $|s|<\check{n}$ and $0$ otherwise.
 Moreover, $f_{n,r}(s)$ is identically $0$ if $n$ is odd.
\end{thm}
 
 Remark that $S^{+}(L_n(s))=S(L_n(s))$ if $|s|<\check{n}$ and
 $S(L_n(s))\setminus\{\check{n}^2-s^2\}$ otherwise (see Section~\ref{sec:spectral zeta}).
 As corollaries,
 letting $s=\check{n}$ and $\frac{1}{2}$
 (notice that these are in $I_n$ for all $n\in\bN$), 
 one can respectively obtain the following expressions of
 the determinants of the Laplacian and the conformal Laplacian.

\begin{cor}
 It holds that 
\begin{align}
\label{for:maindn}
 \det\bigl(\Delta_n\bigr)
&=e^{f_{n,1}(\check{n})}
\prod^{n-1}_{l=1}
\Bigl(\G_l(n)\G_l(1)\Bigr)^{-1}\cdot \Bigl(\G_n(n)\G_n(1)\Bigr)^{-2},\\
\label{for:mainyn}
 \det\bigl(Y_n\bigr)
&=\bigl(\frac{n(n-2)}{4}\bigr)^{\d_{n\ge 3}}
\cdot e^{f_{n,1}(\frac{1}{2})}
\prod^{n-1}_{l=1}
\Bigl(\G_l\bigl(\frac{n}{2}+1\bigr)\G_l\bigl(\frac{n}{2}\bigr)\Bigr)^{-1}\cdot 
\Bigl(\G_n\bigl(\frac{n}{2}+1\bigr)\G_n\bigl(\frac{n}{2}\bigr)\Bigr)^{-2}, 
\end{align}
 where $\d_{n\ge 3}:=1$ if $n\ge 3$ and $0$ otherwise.
\end{cor}

 Remark that these factorization formulas with $r=1$ are essentially obtained in
 \cite{QuineChoi1996}.
 We also remark that, from a calculation point of view,
 the formula \eqref{for:maindn} is more easily computable than \eqref{for:vardi}
 because we have the ladder relation $\G_{l}(z+1)=\G_l(z)\G_{l-1}(z)^{-1}$
 for the multiple gamma functions. 
 In fact, using the ladder relation,
 from \eqref{for:maindn} and \eqref{for:mainyn},
 one can respectively obtain explicit expressions of $\det(\Delta_n)$ and $\det(Y_n)$
 in terms of the derivatives of the Riemann zeta function similar to \eqref{for:vardi}. 
 
 The organization of the paper is as follows.
 In Section~\ref{sec:spectral zeta},
 we first study the spectral zeta function $\z_{L_n(s)}(w)$ in detail and 
 show that it is expressed as a linear combination of the ``two-variable Hurwitz zeta function''
 $H_{\a_1,\a_2}(w_1,w_2)$ (Proposition~\ref{prop:anazeta}).
 Section~\ref{sec:MG} is devoted to reviewing the multiple gamma functions,
 which we need throughout the paper.
 In Section~\ref{sec:main},
 we show that the higher depth determinant $\Det_{r}(L_n(s))$
 can be expressed as a product of $I^{d}_{n,r}(s)$,
 which are defined by a derivative of $H_{\a_1,\a_2}(w_1,w_2)$,
 and then prove that these can be written in terms of the multiple gamma functions
 (Proposition~\ref{prop:I}).   
 Simplifying this expression, we obtain the main result (Theorem~\ref{thm:main}).
 Finally, in Section~\ref{sec:r=1}, as corollaries, 
 we study the case $r=1$ precisely and
 derive explicit expressions of $\det(\Delta_n)$ and $\det(Y_n)$
 in terms of the derivatives of the Riemann zeta function (Corollary~\ref{cor:r=1}).

 Throughout the present paper, 
 we denote by $\bC$, $\bR$ and $\bQ$ the fields of
 all complex, all real and all rational numbers, respectively.
 We also use the notation $\bZ$, $\bN$ and $\bN_0$ to denote the sets of
 all rational, all positive and all non-negative integers, respectively.

\section{Spectral zeta functions}
\label{sec:spectral zeta}

 Let $\l_{n,k}$ be the $k$-th eigenvalue of the Laplacian $\Delta_n$ 
 and $m_{n,k}$ the multiplicity of $\l_{n,k}$.
 It is known that $\l_{n,k}$ and $m_{n,k}$ are given by
 $\l_{n,k}=k(k+n-1)$ and $m_{n,k}=\binom{n+k}{n}-\binom{n+k-2}{n}$, respectively.
 Let $s\in I_n:=(-\check{n}-1,\check{n}+1)$ with $\check{n}:=\frac{n-1}{2}$ and
 $L_n(s):=\Delta_n+\check{n}^2-s^2$.
 Notice that, since $\l_{n,k}+\check{n}^2-s^2=(k+\check{n}+s)(k+\check{n}-s)>0$ for all $k\ge 1$,
 we have $S^{+}(L_n(s))=S(L_n(s))$ if $|s|<\check{n}$ and
 $S(L_n(s))\setminus\{\check{n}^2-s^2\}$ otherwise.
 Instead of $\z_{L_n(s)}(w)$, we study the function $\tz_{L_n(s)}(w)$
 defined by 
\begin{align}
\label{for:spectralzetan}
 \tz_{L_n(s)}(w)
:=\sum^{\infty}_{k=1}\frac{m_{n,k}}{\bigl(\l_{n,k}+\check{n}^2-s^2\bigr)^w}
=\sum^{\infty}_{k=0}
\frac{\binom{n+k+1}{n}-\binom{n+k-1}{n}}{(k+\xi^{+}_{n}(s))^w(k+\xi^{-}_{n}(s))^w},
\end{align}
 where $\xi^{\pm}_n(s):=1+\check{n}\pm s$.
 The series converges absolutely for $\Re(w)>\frac{n}{2}$ and defines a holomorphic function in the region.
 From the definition \eqref{def:speczeta} and the above observation, 
 we have 
\begin{equation}
\label{for:z-tz}
 \z_{L_n(s)}(w)
=\frac{\d_{|s|<\check{n}}}{(\check{n}^2-s^2)^w}+\tz_{L_n(s)}(w),
\end{equation}
 where $\d_{|s|<\check{n}}:=1$ if $|s|<\check{n}$ and $0$ otherwise.
 To study analytic properties of $\tz_{L_n(s)}(w)$,
 we employ the following two-variable Hurwitz zeta function
 $H_{\a_1,\a_2}(w_1,w_2)$ studied by Mizuno (\cite{Mizuno2006}):
\[
 H_{\a_1,\a_2}(w_1,w_2):
=\sum^{\infty}_{k=0}
\frac{1}{(k+\a_1)^{w_1}(k+\a_2)^{w_2}}.
\]
 This converges absolutely for $\Re(w_1+w_2)>1$.
 Here, we let $\Re(\a_1)>0$ and $\Re(\a_2)>0$.

\begin{lem}
\label{lem:H}
 Suppose $|\a_2-\a_1|<\min\{1,|\a_2|\}$.
 Then, 
\begin{equation}
\label{for:H-binom}
 H_{\a_1,\a_2}(w_1,w_2)
=\sum^{\infty}_{l=0}\binom{w_1+l-1}{l}(\a_2-\a_1)^l\z(w_1+w_2+l,\a_2).
\end{equation}
 This gives a meromorphic continuation of $H_{\a_1,\a_2}(w_1,w_2)$
 to the whole space $\bC^2$
 with possible singularities on $w_1+w_2=1-l$ for $l\in\bZ_{\ge 0}$.
\end{lem}
\begin{proof}
 The assumption $|\a_2-\a_1|<|\a_2|$ implies that
 $|\frac{\a_2-\a_1}{k+\a_2}|<1$ for all $k\ge 0$. 
 Hence, by the binomial theorem, we have 
\begin{align*}
 H_{\a_1,\a_2}(w_1,w_2)
&=\sum^{\infty}_{k=0}\frac{1}{(k+\a_2)^{w_1+w_2}}
\Bigl(1-\frac{\a_2-\a_1}{k+\a_2}\Bigr)^{-w_1}\\
&=\sum^{\infty}_{k=0}\frac{1}{(k+\a_2)^{w_1+w_2}}
\sum^{\infty}_{l=0}\binom{w_1+l-1}{l}
\Bigl(\frac{\a_2-\a_1}{k+\a_2}\Bigr)^l\\
&=\sum^{\infty}_{l=0}\binom{w_1+l-1}{l}(\a_2-\a_1)^l\z(w_1+w_2+l,\a_2).
\end{align*}
 This shows the expression \eqref{for:H-binom}.
 Moreover, by the assumption $|\a_2-\a_1|<1$ and the fact that 
 $\z(w_1+w_2+l,\a_2)$ is uniformly bounded with respect to the variable $l$,
 one sees that this gives a meromorphic continuation of $H_{\a_1,\a_2}(w_1,w_2)$ to $\bC^2$.
 The rest of the assertion is clear because $\z(w,z)$ has a simple pole at $w=1$.
\end{proof}

\begin{remark}
 In \cite{Mizuno2006}, 
 the expression \eqref{for:H-binom} is obtained
 by establishing a contour integral representation of $H_{\a_1,\a_2}(w_1,w_2)$.
 See also \cite{MatsumotoWeng2002},
 where meromorphic continuations of more generally ``multi-variable Hurwitz zeta functions'' are obtained
 by using the Mellin-Barnes integral.
\end{remark}

 We next write $\tz_{L_{n}(s)}(w)$
 as a linear combination of $H_{\a_1,\a_2}(w_1,w_2)$.
 Define 
\begin{equation}
\label{def:P}
 P_n(z):=\binom{z+\che{n}+1}{n}-\binom{z+\che{n}-1}{n}.
\end{equation}
 Since $P_n(z)=\frac{1}{n!}((z-\check{n}+1)_n-(z-\check{n}-1)_n)$,
 where $(w)_n:=w(w+1)\cdots (w+n-1)$,  
 using the expansion $(w)_n=\sum^{n}_{m=0}(-1)^{n+m}s(n,m)w^m$, 
 where $s(n,m)$ is the Stirling number of the first kind,
 it can be written as 
\begin{equation}
\label{for:P}
 P_n(z)
=\sum^{n-1}_{d=0}T_{n,d}(s)(z-s)^d
=\sum^{n-1}_{d=0}T_{n,d}(-s)(z+s)^d.
\end{equation}
 Here $T_{n,d}(s)$ is the polynomial of degree $n-1-d$ defined by
\[
 T_{n,d}(s)
:=\frac{1}{n!}\sum^{n}_{m=d+1}\binom{m}{d}(-1)^{n+m}s(n,m)
\Bigl((s-\che{n}+1)^{m-d}-(s-\che{n}-1)^{m-d}\Bigr).
\] 
 The polynomial $T_{n,d}(s)$ plays an important rule in our study and 
 will be precisely investigated in Subsection~\ref{subsec:polynomials}.
 Employing $T_{n,d}(s)$, one obtains the following proposition.

\begin{prop}
\label{prop:anazeta}
 Suppose $|s|<\frac{1}{3}$. Then, 
\begin{equation}
\label{for:spectralzetaH}
 \tz_{L_{n}(s)}(w)
=\sum^{n-1}_{d=0}T_{n,d}(s)H_{\xi^{+}_{n}(s),\xi^{-}_{n}(s)}(w,w-d).
\end{equation}
 This gives a meromorphic continuation of $\tz_{L_{n}(s)}(w)$
 to the whole plane $\bC$ with possible simple poles at
 $w=\frac{n-m}{2}$ for $m\in\bZ_{\ge 0}$
 except for non-positive integer points.
\end{prop}
\begin{proof}
 Since $\binom{n+k+1}{n}-\binom{n+k-1}{n}=P_n(k+1+\check{n})=\sum^{n-1}_{d=0}T_{n,d}(s)(k+\xi^{-}_{n}(s))^d$,
 from \eqref{for:spectralzetan}, 
 one immediately obtains the expression \eqref{for:spectralzetaH}.
 Let $\a_1=\xi^{+}_{n}(s)$ and $\a_2=\xi^{-}_{n}(s)$.
 Then, the assumption $|s|<\frac{1}{3}$ implies that
 $s\in I_n$ and $|\a_2-\a_1|<\min\{1,|\a_2|\}$.
 Hence, from Lemma~\ref{lem:H}, it can be written as
\begin{equation}
\label{for:Hss}
 H_{\xi^{+}_{n}(s),\xi^{-}_{n}(s)}(w,w-d)
=\sum^{\infty}_{l=0}\binom{w+l-1}{l}(-2s)^l\z\bigl(2w-d+l,\xi^{-}_{n}(s)\bigr).
\end{equation}
 This gives a meromorphic continuation of $H_{\xi^{+}_{n}(s),\xi^{-}_{n}(s)}(w,w-d)$
 to $\bC$ as a function of $w$
 with possible simple poles at $w=\frac{d+1-l}{2}$ for $l\in\bZ_{\ge 0}$.
 Therefore, from the expression \eqref{for:spectralzetaH},
 $\tz_{L_{n}(s)}(w)$ also admits a meromorphic continuation to $\bC$ 
 with possible simple poles at $w=\frac{n-m}{2}$ for $m\in\bZ_{\ge 0}$.
 We now claim that it is holomorphic at non-positive integer points.
 To see this, it is enough to show that
 the summand of the righthand-side of \eqref{for:Hss} with $l=2r+d-1$
 is holomorphic at $w=1-r$ for $r\in\bN$.
 In fact, this can be seen from the asymptotic formulas 
\begin{align*}
 \binom{w+2r+d-2}{2r+d-1}
&=\frac{(-1)^{r-1}(r-1)!(r+d-1)!}{(2r+d-1)!}(w+r-1)+O\bigl((w+r-1)^2\bigr),\\
 \z\bigl(2w+2r-1,\xi^{-}_{n}(s)\bigr)
&=\frac{1}{2(w+r-1)}+O(1)
\end{align*}
 as $w\to 1-r$.
 This completes the proof of the proposition.
\end{proof}

\section{Multiple gamma functions}
\label{sec:MG}

 In this section,
 we review the multiple gamma functions.

\subsection{Definitions}

 The Barnes multiple zeta function (\cite{Barnes}) is defined by  
\begin{equation*}
%\label{def:multiple_zeta}
 \z_n(w,z)
:=\sum_{m_1,\ldots,m_n\ge 0}\frac{1}{(m_1+\cdots+m_n+z)^w}
\qquad (\Re(w)>n).
\end{equation*} 
 This clearly gives a generalization of the Hurwitz zeta function: $\z_{1}(w,z)=\z(w,z)$.
 It is known that $\z_n(w,z)$ admits a meromorphic continuation to the whole plane $\bC$
 with possible simple poles at $w=1,2,\ldots,n$.
 Using $\z_n(w,z)$, the multiple gamma function $\G_{n,r}(z)$ of depth $r$ is defined by
\begin{equation*}
%\label{def:multiple_gamma}
 \G_{n,r}(z):=
\exp\Bigl(\frac{\p}{\p w}\z_{n}(w,z)\Bigl|_{w=1-r}\Bigr).
\end{equation*}
 In particular,
 we put $\G_{n}(z):=\G_{n,1}(z)$ and $\mG_{r}(z):=\G_{1,r}(z)$.  
 These are respectively called 
 the Barnes multiple gamma function (\cite{Barnes})
 and the Milnor gamma function of depth $r$ (\cite{M}, see also \cite{KOW}).
 Note that, from the Lerch formula
 $\exp(\frac{\p}{\p w}\z(w,z)\bigl|_{w=0})=\frac{\G(z)}{\sqrt{2\pi}}$,
 we have $\G_{1,1}(z)=\G_{1}(z)=\mG_1(z)=\frac{\G(z)}{\sqrt{2\pi}}$.
 The Barnes multiple gamma function is a meromorphic function
 with poles at non-positive integer points. 
 On the other hand,
 the Milnor gamma function in general has branch points at these points.

\subsection{Barnes multiple gamma functions}

 It is easy to see that $\z_n(w,z)$ can be written as a linear combination
 of the Hurwitz zeta functions:
\begin{equation}
\label{for:Barneszeta}
 \z_n(w,z)
=\sum^{\infty}_{m=0}\binom{m+n-1}{n-1}\frac{1}{(m+z)^{w}}
=\sum^{n-1}_{k=0}b_{n,k}(z)\z(w-k,z),
\end{equation}
 where, for $0\le k\le n-1$, $b_{n,k}(z)$ is the polynomial defined by 
\begin{equation}
\label{def:b}
 \binom{m+n-1}{n-1}=\sum^{n-1}_{k=0}b_{n,k}(z)(m+z)^k.
\end{equation}
 More explicitly, it can be written as  
 $b_{n,k}(z)=\frac{(-1)^{n-1-k}}{(n-1)!}\sum^{n-1}_{j=k}\binom{j}{k}s(n,j+1)z^{j-k}$.
 The expression \eqref{for:Barneszeta} immediately yields the equation  
\[
 \G_{n}(z)=e^{\sum^{n-1}_{k=0}b_{n,k}(z)\z'(-k,z)}.
\]
 In particular, one obtains
\begin{align}
\label{for:G1}
 \G_n(1)
&=e^{\sum^{n-1}_{k=0}b_{n,k}(1)\z'(-k)}, \\
\label{for:G1/2}
 \G_n\bigl(\frac{1}{2}\bigr)
&=2^{-\sum^{n-1}_{k=0}b_{n,k}(\frac{1}{2})\frac{2^{-k}B_{k+1}}{k+1}}
e^{\sum^{n-1}_{k=0}b_{n,k}(\frac{1}{2})(2^{-k}-1)\z'(-k)}, 
\end{align}
 where $B_k$ is the Bernoulli number
 defined by the generating function $\frac{te^{t}}{e^t-1}=\sum^{\infty}_{k=0}B_k\frac{t^k}{k!}$.
 Notice that, to obtain the second formula,
 we have used the equation $\z(w,\frac{1}{2})=(2^w-1)\z(w)$
 and the formula $\z(1-k)=-\frac{B_k}{k}$. 
 Moreover, one can claim that 
 the special values of $\G_n(z)$ at both positive integer and half integer points
 are also expressed in terms of $b_{n,k}(z)$
 and the derivatives of $\z(w)$ at non-positive integer points.
 Actually, from the formula \eqref{for:G1} and \eqref{for:G1/2},
 this can be seen from the equation
\begin{equation}
\label{for:ladderformula}
 \G_n(z+m)=\prod^{m}_{l=0}\G_{n-l}(z)^{(-1)^l\binom{m}{l}}
\qquad (0\le m\le n-1),
\end{equation}
 which is obtained by induction on $m$
 from the ladder relation (the case $m=1$) 
\begin{equation}
\label{for:ladder}
 \G_{n}(z+1)=\G_n(z)\G_{n-1}(z)^{-1}.
\end{equation}
 Here, we put $\G_0(z):=z^{-1}$.

% The Stirling formula of $\G_n(z)$ can be obtained as follows;
%
%\begin{lem}
% For any $N\in\bN$, we have
%\begin{align}
% \log{\G_n(z)}
%&=\sum^{n-1}_{m=0}\frac{B_{n,m}}{m!}\Bigl[\frac{(-1)^{n-m}}{(n-m)!}H(n,m)z^{n-m}
%-\frac{(-1)^{n-m}}{(n-m)!}z^{n-m}\log{z}\Bigr]-\frac{B_{n,n}}{n!}\log{z}\\
%&\ \ \ +\sum^{N+n-1}_{m=n+1}\frac{B_{n,m}}{m!}(m-n-1)!z^{n-m}+O(z^{-N})
%\nonumber
%\end{align}
% as $z\to+\infty$.
%\end{lem}

\subsection{Milnor gamma functions}

 It is shown in \cite{KOW} that
 the Milnor gamma function can be written as a product of the 
 Barnes multiple gamma functions.
 In fact, it can be written as    
\begin{equation}
\label{for:reduce-mG}
 \mG_{r}(z)=\prod^{r}_{l=1}\G_{l}(z)^{c_{r,l}(z)},
\end{equation}
 where, for $1\le l\le r$, $c_{r,l}(z)$ is the polynomial in $z$
 defined by $c_{r,l}(z):=\sum^{l-1}_{k=0}\binom{l-1}{k}(-1)^{k}(z-k-1)^{r-1}$.
 In other words, it is defined by the generating function
\begin{equation}
\label{for:genec}
 (T+z)^{r-1}=\sum^{r}_{j=1}c_{r,j}(z)\binom{T+j-1}{j-1}.
\end{equation}
 For example,
 $c_{r,r}(z)=(r-1)!$, $c_{r,r-1}(z)=\frac{1}{2}(2z-r)(r-1)!$, \ldots,
 and $c_{r,1}(z)=z^{r-1}$.

\section{Higher depth determinants}
\label{sec:main}

\subsection{Main results}
\label{subsec:mainresult}

 Let 
\[
 \widetilde{\Det}_r\bigl(L_n(s)\bigr)
:=\exp\Bigl(-\frac{\p}{\p w}\tz_{L_{n}(s)}(w)\Bigl|_{w=1-r}\Bigr).
\]
 Then, from the equation \eqref{for:z-tz}, we have 
\begin{equation}
\label{for:DettildeDet}
 \Det_r\bigl(L_n(s)\bigr)
=\bigl(\check{n}^2-s^2\bigr)^{(\check{n}^2-s^2)^{r-1}\d_{|s|<\check{n}}}\cdot
 \widetilde{\Det}_r\bigl(L_n(s)\bigr).
\end{equation}
 Let $|s|<\frac{1}{3}$. 
 Then, from Proposition~\ref{prop:anazeta},
 we have
\begin{equation}
\label{for:HD1}
 \widetilde{\Det}_r\bigl(L_n(s)\bigr)
=\prod^{n-1}_{d=0}I^{d}_{n,r}(s)^{T_{n,d}(s)}, 
\end{equation}
 where
\begin{equation*}
 I^{d}_{n,r}(s)
:=\exp\Bigl(-\frac{\p}{\p w}H_{\xi^{+}_{n}(s),\xi^{-}_{n}(s)}(w,w-d)\Bigl|_{w=1-r}\Bigr).
\end{equation*}
 Hence, our next task is to calculate the function $I^{d}_{n,r}(s)$ explicitly.
 Let $H(n):=\sum^{n}_{j=1}\frac{1}{j}$ and 
 $H(m,n):=\sum^{n}_{j=m}\frac{1}{j}=H(n)-H(m-1)$.
 We understand that $H(n)=0$ if $n\le 0$ and $H(m,n)=0$ if $m>n$, respectively.
 The next proposition says that 
 $I^{d}_{n,r}(s)$ can be written as a product of the Milnor gamma functions.

\begin{prop}
\label{prop:I}
 It holds that 
\begin{align}
\label{for:I2}
 I^{d}_{n,r}(s)
&=\exp\Bigl(\frac{(-1)^{r+d}(r-1)!}{2(r+d)_r}H(r,r+d-1)(2s)^{2r+d-1}\Bigr)\\
&\ \ \ \times
 \prod^{2r+d-1}_{k=r}\mG_{k}\bigl(\xi^{+}_{n}(s)\bigr)^{(-1)^{k+d}\binom{r+d-1}{k-r}(2s)^{2r+d-k-1}}\cdot 
 \prod^{2r+d-1}_{k=r+d}\mG_{k}\bigl(\xi^{-}_{n}(s)\bigr)^{-\binom{r-1}{2r+d-k-1}(2s)^{2r+d-k-1}}.\nonumber
\end{align}
\end{prop}
\begin{proof}
 The expression \eqref{for:Hss} can be written as
\begin{align*}
& H_{\xi^{+}_{n}(s),\xi^{-}_{n}(s)}(w,w-d)\\
&\ \ =\Biggl(
\sum_{l=0}+\sum^{r-1}_{l=1}
+\sum^{2r+d-2}_{l=r}+\sum_{l=2r+d-1}+\sum^{\infty}_{l=2r+d}\Biggr)
\binom{w+l-1}{l}(-2s)^l\z\bigl(2w-d+l,\xi^{-}_{n}(s)\bigr)\\
&\ \ =:H^{1}_{n,r,d}(w,s)+H^{2}_{n,r,d}(w,s)+H^{3}_{n,r,d}(w,s)
+H^{4}_{n,r,d}(w,s)+H^{5}_{n,r,d}(w,s).
\end{align*}
 Let $\wt{H}^{j}_{n,r,d}(s):=-\frac{\p}{\p w}H^{j}_{n,r,d}(w,s)\bigl|_{w=1-r}$.
 For simplicity,
 we write $H^{j}(w,s)=H^{j}_{n,r,d}(w,s)$ and $\wt{H}^{j}(s)=\wt{H}^{j}_{n,r,d}(s)$.
 Then, since $I^{d}_{n,r}(s)=\exp(\sum^{5}_{j=1}\wt{H}^{j}(s))$,
 the expression \eqref{for:I2} follows from the formulas
 \eqref{for:H1}, \eqref{for:H2}, \eqref{for:H3}, \eqref{for:H4} and \eqref{for:H5}
 in the next subsection.
\end{proof}

 From the equations \eqref{for:HD1} and \eqref{for:I2},
 we have
\begin{equation}
\label{for:Detr1}
 \widetilde{\Det}_{r}\bigl(L_n(s)\bigr)
=e^{f_{n,r}(s)}
\prod^{2r+n-2}_{k=r}
\mG_k\bigl(\xi^{+}_{n}(s)\bigr)^{\a^{+}_{n,r}(s,k)}
\mG_k\bigl(\xi^{-}_{n}(s)\bigr)^{\a^{-}_{n,r}(s,k)}.
\end{equation}
 Here, $f_{n,r}(s)$ is the polynomial of degree $2r+n-1$ defined by
\begin{align*}
 f_{n,r}(s)
&:=\sum^{n-1}_{d=1}\frac{(-1)^{r+d}(r-1)!}{2(r+d)_r}H(r,r+d-1)
(2s)^{2r+d-1}T_{n,d}(s)
\end{align*}  
 and, for $r\le k\le 2r+n-2$, $\a^{\pm}_{n,r}(s,k)$ are the polynomials in $s$
 respectively defined by 
\begin{align}
\label{for:ap}
 \a^{+}_{n,r}(s,k)
&:=\sum^{n-1}_{d=\max\{0,k-2r+1\}}(-1)^{k+d}
\binom{r+d-1}{k-r}(2s)^{2r+d-k-1}T_{n,d}(s),\\
\label{for:am}
 \a^{-}_{n,r}(s,k)
&:=-\sum^{\min\{n-1,k-r\}}_{d=\max\{0,k-2r+1\}}
\binom{r-1}{2r+d-k-1}(2s)^{2r+d-k-1}T_{n,d}(s).
\end{align}
 
\begin{lem}
\label{lem:apm}
 Let $A^{\pm}_{n,r}(s,X):=\sum^{2r+n-2}_{k=r}\a^{\pm }_{n,r}(s,k)X^{k-1}$.
 Then, 
\begin{equation}
\label{for:gene-alpha}
 A^{+}_{n,r}(s,X)=A^{-}_{n,r}(-s,X)
=-P_n(X-s)X^{r-1}(X-2s)^{r-1}.
\end{equation}
 In particular,
$\a^{+}_{n,r}(s,k)=\a^{-}_{n,r}(-s,k)$.
\end{lem}
 This lemma will be proved in Subsection~\ref{subsec:polynomials}.
Hereafter, we write $A_{n,r}(s,X):=A^{+}_{n,r}(s,X)=A^{-}_{n,r}(-s,X)$ 
 and $\a_{n,r}(s,k):=\a^{+}_{n,r}(s,k)=\a^{-}_{n,r}(-s,k)$.
 Moreover, from the formula \eqref{for:reduce-mG}, 
 the expression \eqref{for:Detr1} can be rewritten as 
\begin{equation}
\label{for:Detr2}
 \widetilde{\Det}_{r}\bigl(L_n(s)\bigr)
=e^{f_{n,r}(s)}
\prod^{2r+n-2}_{l=1}
\G_l\bigl(\xi^{+}_{n}(s)\bigr)^{\b^{+}_{n,r}(s,l)}
\G_l\bigl(\xi^{-}_{n}(s)\bigr)^{\b^{-}_{n,r}(s,l)},
\end{equation}
 where, for $1\le l\le 2r+n-2$,
 $\b^{\pm}_{n,r}(s,l)$ are the polynomials in $s$ defined by
\[
 \b^{\pm}_{n,r}(s,l)
:=\sum^{2r+n-2}_{k=\max\{r,l\}}c_{k,l}\bigl(\xi^{\pm}_{n}(s)\bigr)\a_{n,r}(\pm s,k).
\]

\begin{lem}
\label{lem:bpm}
 Let $B^{\pm}_{n,r}(s,Y):=\sum^{2r+n-2}_{l=1}\b^{\pm}_{n,r}(s,l)\binom{Y+l-1}{l-1}$.
 Then,
\begin{align}
\label{for:gene-b}
 B^{\pm}_{n,r}(s,Y)
=-P_n(Y+1+\check{n})\bigl(Y+\xi^{+}_{n}(s)\bigr)^{r-1}\bigl(Y+\xi^{-}_{n}(s)\bigr)^{r-1}.
\end{align} 
 In particular, $\b^{+}_{n,r}(s,l)=\b^{-}_{n,r}(s,l)$, which is an even polynomial.
\end{lem}
 This will be also proved in Subsection~\ref{subsec:polynomials}.
 We also write $B_{n,r}(s,Y):=B^{+}_{n,r}(s,Y)=B^{-}_{n,r}(s,Y)$
 and $\b_{n,r}(s,l):=\b^{+}_{n,r}(s,l)=\b^{-}_{n,r}(s,l)$.
 Now, from the expression \eqref{for:Detr2} and Lemma~\ref{lem:bpm},
 we have
\begin{equation}
\label{for:Detr3}
 \widetilde{\Det}_{r}\bigl(L_n(s)\bigr)
=e^{f_{n,r}(s)}
\prod^{2r+n-2}_{l=1}
\Bigl(\G_l\bigl(\xi^{+}_{n}(s)\bigr)\G_l\bigl(\xi^{-}_{n}(s)\bigr)\Bigr)^{\b_{n,r}(s,l)}.
\end{equation}
 Since the Barnes multiple gamma function $\G_l(z)$ is a meromorphic function in $\bC$
 having poles at $z=-m$ for $m\in\bZ_{\ge 0}$,
 the expression \eqref{for:Detr3} gives an analytic continuation of
 $\widetilde{\Det}_{r}(L_n(s))$ to the region  
 $\bC\setminus ((-\infty,-\che{n}-1]\cup [\che{n}+1,\infty))$.
 In particular, it is valid for all $s\in I_n$.
 Therefore, from \eqref{for:DettildeDet}, one obtains the following result.

\begin{thm}
\label{thm:main}
 Let $s\in I_n=(-\check{n}-1,\check{n}+1)$. Then, 
\begin{equation}
\label{for:DetrB}
 \Det_{r}\bigl(L_n(s)\bigr)
=\bigl(\check{n}^2-s^2\bigr)^{(\check{n}^2-s^2)^{r-1}\d_{|s|<\check{n}}}\cdot
e^{f_{n,r}(s)}
\prod^{2r+n-2}_{l=1}
\Bigl(\G_l\bigl(\xi^{+}_{n}(s)\bigr)\G_l\bigl(\xi^{-}_{n}(s)\bigr)\Bigr)^{\b_{n,r}(s,l)}.
\end{equation}
 In particular, 
 letting $s=\check{n}$ and $s=\frac{1}{2}$, it follows, respectively, that
\begin{align}
\label{for:DetrDn}
 \Det_r(\Delta_n)
&=e^{f_{n,r}(\che{n})}
\prod^{2r+n-2}_{l=1}
\bigl(\G_l(n)\G_l(1)\bigr)^{\b_{n,r}(\check{n},l)},\\
\label{for:DetrYn}
 \Det_r(Y_n)
&=\bigl(\frac{n(n-2)}{4}\bigr)^{(\frac{n(n-2)}{4})^{r-1}\d_{n\ge 3}}\cdot
e^{f_{n,r}(\frac{1}{2})}
\prod^{2r+n-2}_{l=1}
\Bigl(\G_l\bigl(\frac{n}{2}+1\bigr)\G_l\bigl(\frac{n}{2}\bigr)\Bigr)^{\b_{n,r}(\frac{1}{2},l)},
\end{align}
 where $\d_{n\ge 3}:=1$ if $n\ge 3$ and $0$ otherwise.
\qed
\end{thm}

\subsection{Calculations of $\wt{H}^{j}(s)$}
\label{subsec:H}
 
 In this subsection,
 to give a complete proof of Proposition~\ref{prop:I},
 we explicitly calculate the functions $\wt{H}^{j}(s)$ for $j=1,2,3,4,5$.
 We need the following lemma in the subsequent discussions.

\begin{lem}
\label{lem:asymptotics}
 Write $w'=w+r-1$. Then, as $w\to 1-r$, 
\begin{align}
\label{for:binom}
\binom{w+l-1}{l}=
\begin{cases}
 \DS{(-1)^l\binom{r-1}{l}\Bigl(1-H(r-l,r-1)w'+O(w'^2)\Bigr)}
 & \textrm{$(1\le l<r)$},\\[7pt]
 \DS{\frac{(-1)^{r+1}}{r}\binom{l}{r}^{-1}
 \Bigl(w'-H(l-r+1,r-1)w'^2+O(w'^3)\Bigl)}
 & \textrm{$(r\le l<2r-1)$},\\[7pt]
 \DS{\frac{(-1)^{r+1}}{r}\binom{l}{r}^{-1}
 \Bigl(w'+H(r,l-r)w'^2+O(w'^3)\Bigl)}
 & \textrm{$(l\ge 2r-1)$}.
\end{cases}
\end{align}
 \qed
\end{lem}

 We also use the following well-known formulas concerning the Hurwitz zeta function:
\begin{align}
\label{for:zetaBernoulli}
 \z(1-k,z)
&=-\frac{B_k(z)}{k} \qquad (k\in\bN), \\
\label{for:Hurwitz-expansion}
 \z(s,z)
&=\frac{1}{s-1}-\psi(z)+O(s-1)
\qquad (s\to 1).
\end{align}
 Here, $B_k(z)$ is the Bernoulli polynomial defined by the generating function
 $\frac{te^{tz}}{e^t-1}=\sum^{\infty}_{k=0}B_k(z)\frac{t^k}{k!}$ and 
 $\psi(z):=\frac{d}{dz}\log{\G(z)}=\frac{\G'}{\G}(z)$ is the digamma function.

\subsubsection{$\wt{H}^{1}(s)$}

 By definition, we have $H^{1}(w,s)=\z(2w-d,\xi^{-}_{n}(s))$. Hence  
\begin{equation}
\label{for:H1}
 \wt{H}^{1}(s)=-2\log{\mG_{2r+d-1}\bigl(\xi^{-}_{n}(s)\bigr)}.
\end{equation}

\subsubsection{$\wt{H}^{2}(s)$ and $\wt{H}^{3}(s)$}

 From the formulas \eqref{for:binom} and \eqref{for:zetaBernoulli},
 it is not difficult to obtain  
\begin{align}
\label{for:H2}
 \wt{H}^{2}(s)
&=-\sum^{r-1}_{l=1}\binom{r-1}{l}\frac{H(r-l,r-1)B_{2r+d-l-1}(\xi^{-}_{n}(s))}{2r+d-l-1}(2s)^l\\
&\ \ \ -2\sum^{2r+d-2}_{k=r+d}\binom{r-1}{2r+d-k-1}
(2s)^{2r+d-k-1}\log{\mG_{k}\bigl(\xi^{-}_{n}(s)\bigr)},\nonumber\\
\label{for:H3}
 \wt{H}^{3}(s)
&=\frac{(-1)^{r+1}}{r}\sum^{2r+d-2}_{l=r}
\binom{l}{r}^{-1}\frac{(-1)^lB_{2r+d-l-1}(\xi^{-}_{n}(s))}{2r+d-l-1}(2s)^{l}.
\end{align}

\subsubsection{$\wt{H}^{4}(s)$}

 By definition, we have 
 $H^{4}(w,s)=\binom{w+2r+d-2}{2r+d-1}(-2s)^{2r+d-1}\z(2w+2r-1,\xi^{-}_{n}(s))$.
 Hence, from the formulas \eqref{for:binom} with $l=2r+d-1 (\ge 2r-1)$
 and \eqref{for:Hurwitz-expansion},  
 we have 
\begin{equation}
\label{for:H4}
 \wt{H}^{4}(s)
=-\frac{(-1)^{r+d}(r-1)!}{(r+d)_r}
\Bigl(\frac{1}{2}H(r,r+d-1)-\psi\bigl(\xi^{-}_{n}(s)\bigr)\Bigr)(2s)^{2r+d-1}. 
\end{equation}

\subsubsection{$\wt{H}^{5}(s)$}

 From the formula \eqref{for:binom}, we have
\begin{equation}
\label{for:H5-1}
 \wt{H}^{5}(s)
=(-1)^r(r-1)!R_{r}\bigl(2s,\xi^{-}_{n}(s);2r+d-2\bigr), 
\end{equation}
 where, for $m\ge r-1$, $R_r(t,z;m)$ is the following series
 involving the Hurwitz zeta function:
\[
 R_r(t,z;m):
=\sum^{\infty}_{l=2}\frac{\z(l,z)}{(l+m-r+1)_r}(-t)^{l+m}.
\]
 To obtain an explicit expression of $\wt{H}^{5}(s)$,
 we have to study the function $R_r(t,z;m)$.
 Using the equation 
 $\sum^{\infty}_{l=2}\z(l,z)t^{l-1}=-\psi(z-t)+\psi(z)$ for $|t|<|z|$
 (see $(5)$, p\,159 in \cite{SrivastavaChoi2001}),
 we have 
\begin{align*}
 \frac{\p^j}{\p t^j}R_{r}(0,z;m)
&=0 \qquad (0\le j\le r-1),\\
 \frac{\p^r}{\p t^r}R_{r}(t,z;m)
&=(-1)^{m+1}t^{m-r+1}\bigl(\psi(z)-\psi(z+t)\bigr).
\end{align*}
 These show that 
\begin{equation}
\label{for:R1}
 R_{r}(t,z;m)
=\frac{(-1)^{m+1}\psi(z)}{(m-r+2)_r}t^{m+1}
+(-1)^m\Phi^{m-r+1}_{r}(t,z),
\end{equation} 
 where, for $m\in\bN$,
\[
 \Phi^{m}_{r}(t,z)
:=\underbrace{\int^{t}_{0}\int^{\xi_{r}}_0\cdots\int^{\xi_2}_{0}}_{r}
{\xi^{m}_1}\psi(\xi_1+z)d\xi_1\cdots d\xi_{r}.
\]
 Therefore, from the formulas \eqref{for:H5-1}, \eqref{for:R1} and Theorem~\ref{thm:phi} below 
 with $t=2s$, $z=\xi^{-}_{n}(s)$ and $m=2r+d-2$,
 we obtain the following expression of $\wt{H}^{5}(s)$:
\begin{align}
\label{for:H5}
 \wt{H}^{5}(s)
&=\sum^{r-1}_{l=1}\binom{r-1}{l}\frac{H(r-l,r-1)B_{2r+d-l-1}(\xi^{-}_{n}(s))}{2r+d-l-1}(2s)^l\\
&\ \ \ +\sum^{2r+d-2}_{k=r+d}\binom{r-1}{2r+d-k-1}
(2s)^{2r+d-k-1}\log{\mG_{k}\bigl(\xi^{-}_{n}(s)\bigr)}\nonumber\\
&\ \ \ -\frac{(-1)^{r+1}}{r}\sum^{2r+d-2}_{l=r}
\binom{l}{r}^{-1}\frac{(-1)^lB_{2r+d-l-1}(\xi^{-}_{n}(s))}{2r+d-l-1}(2s)^{l}\nonumber\\
&\ \ \ +\frac{(-1)^{r+d}(r-1)!}{(r+d)_r}\Bigl(H(r,r+d-1)-\psi\bigl(\xi^{-}_{n}(s)\bigr)\Bigr)(2s)^{2r+d-1}\nonumber\\
&\ \ \ +\sum^{2r+d-1}_{k=r}(-1)^{k+d}\binom{r+d-1}{k-r}(2s)^{2r+d-k-1}\log{\mG_{k}\bigl(\xi^{+}_{n}(s)\bigr)}.\nonumber
\end{align}

\begin{thm}
\label{thm:phi}
 It holds that 
\begin{equation}
\label{for:iterated}
 \Phi^{m}_{r}(t,z)
=\sum^{m+r}_{k=r}\frac{(-1)^{k+r}}{(r-1)!}
\binom{m}{k-r}t^{m+r-k}\log{\mG_k(t+z)}
+P^{m}_{r}(t,z),
\end{equation}
 where $P^{m}_{r}(t,z)$ is the polynomial in $t$ of degree $m+r$ defined by 
\begin{align*}
 P^{m}_{r}(t,z)
&=\frac{1}{(m+1)_r}H(r,m)t^{m+r}
+\frac{(-1)^{m+r-1}}{r!}\sum^{m+r-1}_{l=r}
\binom{l}{r}^{-1}\frac{(-1)^{l}B_{m+r-l}(z)}{m+r-l}t^l\\
&\ \ \
 +\frac{(-1)^{m+1}}{(r-1)!}\sum^{r-1}_{l=1}\binom{r-1}{l}\frac{H(r-l,r-1)B_{m+r-l}(z)}{m+r-l}t^{l}\nonumber\\
&\ \ \ 
 +\frac{(-1)^{m+1}}{(r-1)!}\sum^{r-1}_{k=0}\binom{r-1}{k}t^k\log{\mG_{m+r-k}(z)}.\nonumber
\end{align*}
\end{thm}

 To prove the theorem, 
 we need the following formulas.

\begin{lem}
 For $r\ge 2$ and $m\ge r-1$,\begin{align}
\label{for:I}
 I_r(k)
&:=\sum^{k-1}_{j=r-1}\frac{(j-1)!}{(j-r+1)!}
=(r-2)!\binom{k-1}{k-r} \qquad (k\ge r).\\
\label{for:J}
 J_r(k)
&:=\sum^{k-1}_{j=r-1}\frac{(j-1)!}{(j-r+1)!}H(j,k-1)
=\frac{(r-2)!}{r-1}\binom{k-1}{k-r} \qquad (k\ge r).\\
\label{for:K}
 K^{m}_{r}(l)
&:=\sum^{m+r}_{k=\max\{r,l\}}\binom{m}{k-r}\binom{k}{l}\frac{(-1)^k}{k}=
\begin{cases}
 \DS{(-1)^r\frac{m!(r-1)!}{(m+r)!}} & (l=0), \\
 \DS{0} & (1\le l\le m), \\
 \DS{\frac{(-1)^{m+r}}{l}\binom{r-1}{l-1-m}} & (m+1\le l\le m+r).
\end{cases}
\end{align}
\end{lem}
\begin{proof}
 Since $I_r(k)=(r-2)!\sum^{k-r}_{j=0}\binom{j+r-2}{j}$,
 the equation \eqref{for:I} is immediately obtained from
 the formula $\sum^{b}_{j=0}\binom{j+a}{j}=\binom{a+b+1}{b}$.
 We next show the equation \eqref{for:J}.
 Consider the generating function $\sum^{\infty}_{k=r}J_r(k)x^k$.
 Changing the order of the summations, we have 
\begin{align*}
 \sum^{\infty}_{k=r}J_r(k)x^k
&=\sum^{\infty}_{j=r-1}\frac{(j-1)!}{(j-r+1)!}\sum^{\infty}_{k=j+1}\sum^{k-1}_{l=j}\frac{1}{l}x^k
=\sum^{\infty}_{j=r-1}\frac{(j-1)!}{(j-r+1)!}\sum^{\infty}_{l=j}\frac{1}{l}\sum^{\infty}_{k=l+1}x^k\\
&=\frac{x}{1-x}\sum^{\infty}_{l=r-1}I_{r}(l+1)\frac{x^{l}}{l}
=\frac{(r-2)!x}{1-x}\sum^{\infty}_{l=0}\binom{l+r-1}{l}\frac{x^{l+r-1}}{l+r-1}\\
&=\frac{(r-2)!}{r-1}x^{r}(1-x)^{-r}
=\sum^{\infty}_{k=r}\frac{(r-2)!}{r-1}\binom{k-1}{k-r}x^{k}.
\end{align*} 
 Notice that the fifth equality follows from
\[
 \sum^{\infty}_{l=0}\binom{l+r-1}{l}\frac{x^{l+r-1}}{l+r-1}
=\int^{x}_{0}t^{r-2}(1-t)^{-r}dt=\frac{1}{r-1}x^{r-1}(1-x)^{-r+1}.
\]
 Therefore, comparing the coefficient of $x^{k}$,
 one obtains the desired formula.
 We finally show the equation \eqref{for:K}.
 When $l=0$, the formula is obtained as     
\[
 K^{m}_{r}(0)
=\sum^{m}_{k=0}\binom{m}{k}\frac{(-1)^{k+r}}{k+r}
=(-1)^{r}\int^{1}_{0}x^{r-1}(1-x)^mdx
=(-1)^{r}\frac{m!(r-1)!}{(m+r)!}.
\]
 The other cases are easily proved by the fact
\[
  K^{m}_{r}(l)
=\frac{(-1)^{r}}{l!}\frac{d^{l-1}}{dx^{l-1}}\Bigl(x^{r-1}(1-x)^{m}\Bigr)\Bigl|_{x=1}
\]
 together with the Leibniz rule. This completes the proof of lemma.
\end{proof}

 We now give the proof of Theorem~\ref{thm:phi}.

\begin{proof}
[Proof of Theorem~\ref{thm:phi}]
 We prove the formula \eqref{for:iterated} by induction on $r$.
 The case $r=1$ has been already proved in 
 \cite[Corollary~$3.17$]{KurokawaWakayamaYamasaki}.
 Suppose that it holds for $r-1$. 
 Then, we have   
\begin{align}
\label{for:phi-middle}
 \Phi^{m}_{r}(t,z)
&=\int^{t}_{0}\Phi^{m}_{r-1}(\xi,z)d\xi\\
&=\sum^{m+r-1}_{j=r-1}\frac{(-1)^{j+r-1}}{(r-2)!}
\binom{m}{j-r+1}\int^{t}_{0}\xi^{m+r-j-1}\log{\mG_j(\xi+z)}d\xi+\int^{t}_{0}P^{m}_{r-1}(\xi,z)d\xi.
\nonumber
\end{align}
 Here, from \cite[Lemma~$3.16$]{KurokawaWakayamaYamasaki} with $n=1$,
 it holds that 
\begin{align}
\label{for:general-int_mG}
 \int^{t}_{0}\xi^m\log{\mG_{r}(\xi+z)}d\xi
=\sum^{m+1}_{l=1}\frac{(-1)^{l-1}m!(r-1)!}{(m+1-l)!(r+l-1)!}t^{m+1-l}\log{\mG_{r+l}(t+z)}
+P_{r}(t,z;m),
\end{align} 
 where $P_{r}(t,z;m)$ is the polynomial of degree $1+r+m$ in $t$
 given by  
\begin{align}
\label{for:Ppolynomial1}
 P_{r}(t,z;m)
&=\sum^{m+1}_{l=1}
\frac{(-1)^lm!(r-1)!H(r,r+l-1)}{(m+1-l)!(r+l)!}
t^{m+1-l}B_{r+l}(t+z)\\
&\ \ \ +\frac{(-1)^{m+1}}{m+1}\binom{m+r}{m+1}^{-1}
\Bigl(\log{\mG_{m+r+1}(z)}-\frac{H(r,m+r)}{m+r+1}B_{m+r+1}(z)\Bigr).
\nonumber
\end{align} 
 Hence, using the formula \eqref{for:general-int_mG}
 and replacing $m$ with $m+r-j-1$ and $r$ with $j$,
 from \eqref{for:phi-middle},
 we have
\begin{align*}
 \Phi^{m}_{r}(t,z) 
&=\frac{(-1)^{r-1}}{(r-2)!}\sum^{m+r-1}_{j=r-1}\sum^{m+r-j}_{l=1}
\frac{(-1)^{j+l-1}m!(j-1)!t^{m+r-j-l}\log{\mG_{j+l}(t+z)}}{(j-r+1)!(m+r-j-l)!(j+l-1)!}\\
&\ \ \ +\frac{(-1)^{r-1}}{(r-2)!}\sum^{m+r-1}_{j=r-1}
\frac{(-1)^jm!}{(j-r+1)!(m+r-j-1)!}P_{j}(t,z;m+r-j-1)
+\int^{t}_{0}P^{m}_{r-1}(\xi,z)d\xi.\nonumber
\end{align*}
 Write this as, say, $\Phi^{m}_{r}(t,z)=A_1+A_2+A_3$.

 At first, writing $j+l=k$ and using the formula \eqref{for:I}, we have  
\begin{align}
\label{for:A1}
 A_1
&=\frac{(-1)^{r-1}}{(r-2)!}\sum^{m+r}_{k=r}
\frac{(-1)^{k-1}m!I_r(k)}{(m+r-k)!(k-1)!}t^{m+r-k}\log{\mG_{k}(t+z)}\\
&=\sum^{m+r}_{k=r}\frac{(-1)^{k+r}}{(r-1)!}\binom{m}{k-r}t^{m+r-k}\log{\mG_{k}(t+z)}.\nonumber
\end{align}
 We next calculate $A_2$.
 From \eqref{for:Ppolynomial1}, we have 
\begin{align*}
 A_2
&=\frac{(-1)^{r-1}}{(r-2)!}\sum^{m+r-1}_{j=r-1}\sum^{m+r-j}_{l=1}
\frac{(-1)^{j+l}m!(j-1)!H(j,j+l-1)}{(j-r+1)!(m+r-j-l)!(j+l)!}t^{m+r-j-l}B_{j+l}(t+z)\\
&\ \ \
 +\frac{(-1)^{r-1}}{(r-2)!}\sum^{m+r-1}_{j=r-1}\frac{(-1)^{m+r}m!(j-1)!}{(j-r+1)!(m+r-1)!}
\Bigl(\log{\mG_{m+r}(z)}-\frac{H(j,m+r-1)}{m+r}B_{m+r}(z)\Bigr).
\nonumber
\end{align*}
 Writing $j+l=k$ again in the first sum and using the formulas \eqref{for:I} and \eqref{for:J},
 we have 
\begin{align*}
 A_2
&=\frac{(-1)^{r-1}}{(r-2)!}\sum^{m+r}_{k=r}\frac{(-1)^{k}m!J_r(k)}{(m+r-k)!k!}t^{m+r-k}B_{k}(t+z)\\
&\ \ \
 +\frac{(-1)^{r-1}}{(r-2)!}\frac{(-1)^{m+r}m!I_r(m+r)}{(m+r-1)!}\log{\mG_{m+r}(z)}+\frac{(-1)^{r}}{(r-2)!}\frac{(-1)^{m+r}m!J_r(m+r)}{(m+r)!}B_{m+r}(z)\\
&=\frac{1}{r-1}\frac{(-1)^{r-1}}{(r-1)!}\sum^{m+r}_{k=r}\binom{m}{k-r}\frac{(-1)^{k}}{k}t^{m+r-k}B_{k}(t+z)\\
&\ \ \ +\frac{(-1)^{m+1}}{(r-1)!}\log{\mG_{m+r}(z)}
-\frac{1}{r-1}\frac{(-1)^{m+1}}{(r-1)!}\frac{B_{m+r}(z)}{m+r}.
\nonumber
\end{align*}
 Moreover, from the equation $B_{k}(t+z)=\sum^{k}_{l=0}\binom{k}{l}t^{k-l}B_{l}(z)$, 
 changing the order of the summations in the first sum
 and using the formula \eqref{for:K},
 we have
\begin{align}
\label{for:A2}
 A_2
&=\frac{1}{r-1}\frac{(-1)^{r-1}}{(r-1)!}\sum^{m+r}_{l=0}K^{m}_{r}(l)t^{m+r-l}B_{l}(z)\\
&\ \ \  +\frac{(-1)^{m+1}}{(r-1)!}\log{\mG_{m+r}(z)}
-\frac{1}{r-1}\frac{(-1)^{m+1}}{(r-1)!}\frac{B_{m+r}(z)}{m+r}\nonumber\\
&=-\frac{1}{r-1}\frac{1}{(m+1)_r}t^{m+r}+\frac{1}{r-1}\frac{(-1)^{m+1}}{(r-1)!}
\sum^{r-1}_{l=1}\binom{r-1}{l}\frac{B_{m+r-l}(z)}{m+r-l}t^l\nonumber\\
&\ \ \ +\frac{(-1)^{m+1}}{(r-1)!}\log{\mG_{m+r}(z)}.\nonumber
\end{align}
 Finally, by the induction assumption, it is easy to see that 
\begin{align}
\label{for:A3}
 A_3
&=\frac{1}{(m+1)_{r}}H(r-1,m)t^{m+r}+\frac{(-1)^{m+r-1}}{r!}\sum^{m+r-1}_{l=r}
\binom{l}{r}^{-1}\frac{(-1)^{l}B_{m+r-l}(z)}{m+r-l}t^{l}\\
&\ \ \
 +\frac{(-1)^{m+1}}{(r-1)!}\sum^{r-1}_{l=2}\binom{r-1}{l}\frac{H(r-l,r-2)B_{m+r-l}(z)}{m+r-l}t^{l}\nonumber\\
&\ \ \ 
 +\frac{(-1)^{m+1}}{(r-1)!}\sum^{r-1}_{k=1}\binom{r-1}{k}t^{k}\log{\mG_{m+r-k}(z)}.\nonumber
\end{align}
 Therefore, from \eqref{for:A2} and \eqref{for:A3}, 
 one can see that $A_2+A_3=P^{m}_{r}(t,z)$.
 Hence, together with \eqref{for:A1},
 we obtain the desired formula \eqref{for:iterated}.
 This completes the proof of the theorem.  
\end{proof}

\begin{remark}
 Note  that not only the derivatives but also the special values of 
 $\tz_{L_n(s)}(w)$ at non-positive integer points  
 can be calculated from the expression \eqref{for:spectralzetaH}.
 In fact, let $r\in\bN$.
 Then, from \eqref{for:spectralzetaH},
 we have $\tz_{L_n(s)}(1-r)=\sum^{n-1}_{d=0}T_{n,d}(s)\sum^{5}_{j=1}H^{j}(1-r,s)$.
 Here, from \eqref{for:binom}, \eqref{for:zetaBernoulli} and \eqref{for:Hurwitz-expansion}, 
 one obtains 
\begin{align*}
 H^{1}(1-r,s)
&=-\frac{B_{2r+d-1}(\xi^{-}_{n}(s))}{2r+d-1},\\
 H^{2}(1-r,s)
&=-\sum^{r-1}_{l=1}\binom{r-1}{l}\frac{B_{2r+d-l-1}(\xi^{-}_{n}(s))}{2r+d-l-1}(2s)^{l},\\
 H^{4}(1-r,s)
&=\frac{(-1)^{r+1}}{2r}\binom{2r+d-1}{r}^{-1}(-2s)^{2r+d-1}
\end{align*}
 and $H^{3}(1-r,s)=H^{5}(1-r,s)=0$. 
 This shows that 
\begin{align*}
 \tz_{L_n(s)}(1-r)
&=-\sum^{n-1}_{d=0}T_{n,d}(s)
\sum^{r-1}_{l=0}\binom{r-1}{l}\frac{B_{2r+d-l-1}(\xi^{-}_{n}(s))}{2r+d-l-1}(2s)^{l}\\
&\ \ \ +\frac{(-1)^{r+1}}{2r}\sum^{n-1}_{d=0}\binom{2r+d-1}{r}^{-1}T_{n,d}(s)(-2s)^{2r+d-1}. 
\end{align*}
 In particular, $\tz_{L_n(s)}(1-r)\in\bQ[s]$.
 Notice that, if $n$ is odd, then $\tz_{L_n(s)}(1-r)=-(\check{n}^2-s^2)^{r-1}$
 because, from the general theory of spectral zeta functions (see \cite{Rosenberg1997}),
 $\z_{L_n(s)}(1-r)=0$ (for $|s|<\check{n}$).
\end{remark}

\subsection{Key polynomials}
\label{subsec:polynomials}

 In this subsection, 
 we study individually the polynomials which appear in Subsection~\ref{subsec:mainresult}.
 In particular, we give proofs of Lemma~\ref{lem:apm} and \ref{lem:bpm}.

\subsubsection{$T_{n,d}(s)$}

 In what follows,
 we understand that $T_{n,d}(s)=0$ if $d<0$ or $d\ge n$.

\begin{lem}
\label{lem:prop-Tnd}
 $(\mathrm{i})$\ $T_{n,d}(-s)=(-1)^{n+d+1}T_{n,d}(s)$.

 $(\mathrm{ii})$\ $T_{n,d}'(s)=(d+1)T_{n,d+1}(s)$.

 $(\mathrm{iii})$\ $T_{n,d}(\pm s)=\sum^{n-1}_{l=d}\binom{l}{d}(\pm 2s)^{l-d}T_{n,l}(\mp s)$.
\end{lem}
\begin{proof}
 Since $P_n(-z)=(-1)^{n+1}P_n(z)$,
 replacing $z$ with $-z$ in the first equation in \eqref{for:P},
 we have 
\begin{equation*}
%\label{for:-Tnd2}
 P_n(z)
=\sum^{n-1}_{d=0}(-1)^{n+d-1}T_{n,d}(s)(z+s)^d. 
\end{equation*}
 Hence, comparing this equation with the second one in \eqref{for:P},
 we obtain the first claim.  
 The second one is obtained by differentiating the first equation in \eqref{for:P}.
 Furthermore, from \eqref{for:P} again, we have 
\begin{align*}
 \sum^{n-1}_{d=0}T_{n,d}(\pm s)(z\mp s)^d
&=\sum^{n-1}_{l=0}T_{n,l}(\mp s)(z\pm s)^d
=\sum^{n-1}_{l=0}T_{n,l}(\mp s)(z\mp s\pm 2s)^l\\
&=\sum^{n-1}_{l=0}T_{n,l}(\mp s)\sum^{l}_{d=0}\binom{l}{d}(\pm 2s)^{l-d}(z\mp s)^d\\
&=\sum^{n-1}_{d=0}\Biggl[\sum^{n-1}_{l=d}\binom{l}{d}(\pm 2s)^{l-d}T_{n,l}(\mp
 s)\Biggr](z\mp s)^d.
\end{align*}
 This shows the last claim.
\end{proof}

 From  fact $(\mathrm{ii})$ in Lemma~\ref{lem:prop-Tnd},
 if we want to know about the polynomial $T_{n,d}(s)$, 
 then it is enough to study only the case $d=0$ because  
 $T_{n,d}(s)=\frac{1}{d!}T^{(d)}_{n,0}(s)$. 

\begin{prop}
 $\mathrm{(i)}$ 
\begin{align}
\label{for:Tn0}
 T_{n,0}(s)
=\frac{2}{(n-1)!}\frac{s\G(s+\check{n})}{\G(s-\check{n}+1)}
=
\begin{cases}
 2 & (n=1), \\
 2s & (n=2), \\
 \DS{\frac{2}{(n-1)!}\prod^{\frac{n-3}{2}}_{j=0}(s^2-j^2)}
& (n\ge 3\,:\,\textrm{odd}), \\
 \DS{\frac{2s}{(n-1)!}\prod^{\frac{n-2}{2}}_{j=1}\Bigl(s^2-\bigl(\frac{2j-1}{2}\bigr)^2\Bigr)} 
& (n\ge 4\,:\,\textrm{even}).
\end{cases}
\end{align}

 $(\mathrm{ii})$
\begin{equation}
\label{for:geneTnd}
 \sum^{\infty}_{n=d+1}T_{n,d}(s)(2t)^{n-1}
=\frac{2}{d!}(2\,\mathrm{arcsinh}\,t)^d\bigl(t+\sqrt{1+t^2}\bigr)^{2s}. 
\end{equation}
\end{prop}
\begin{proof}
 The equation \eqref{for:Tn0} is easily derived from
 the identity $T_{n,0}(s)=P_n(s)$ using \eqref{def:P}.
 We next show the equation \eqref{for:geneTnd}.
 We first let $d=0$.
 Then, using the formulas (c.f., \cite{GR2007})
\begin{align*}
 \cos\Bigl(x\log{\bigl(t+\sqrt{1+t^2}\bigr)}\Bigr)
&=1+\sum_{n\ge 3\,:\,\textrm{odd}}(-1)^{\frac{n-1}{2}}\frac{1}{(n-1)!}\prod^{\frac{n-3}{2}}_{j=0}\bigl(x^2+(2j)^2\bigr)t^{n-1},\\
 \sin\Bigl(x\log{\bigl(t+\sqrt{1+t^2}\bigr)}\Bigr)
&=xt+\sum_{n\ge 4\,:\,\textrm{even}}(-1)^{\frac{n-2}{2}}\frac{x}{(n-1)!}\prod^{\frac{n-2}{2}}_{j=1}\bigl(x^2+(2j-1)^2\bigr)t^{n-1},
\end{align*}
 we have
\begin{align*}
 2\bigl(t+\sqrt{1+t^2}\bigr)^{2s}
&=2\exp\Bigl(i\bigl(-2si\log{\bigl(t+\sqrt{1+t^2}\bigr)}\bigr)\Bigr)\\
&=2\Bigl(\cos(-2si\log{\bigl(t+\sqrt{1+t^2}\bigr)}\Bigr)
+i\sin\Bigl(-2si\log{\bigl(t+\sqrt{1+t^2}\bigr)})\Bigr)\\
&=\sum^{\infty}_{n=1}T_{n,0}(s)(2t)^{n-1}.
\end{align*}
 Hence, one obtains the claim.
 Notice that, in the last equation, we have used the expression \eqref{for:Tn0}. 
 Now the formula \eqref{for:geneTnd} for $d\ge 1$ can be obtained
 by differentiating this equation $d$ times with respect to the variable $s$
 together with the identity $\mathrm{arcsinh}\,t=\log{(t+\sqrt{1+t^2})}$. 
\end{proof}

\begin{table}[htbp]
\begin{center}
{
{\renewcommand\arraystretch{1.15}
\begin{tabular}{r||c|c|c|c|c|c}
 $d$ & $0$ & $1$ & $2$ & $3$ & $4$ & $5$\\
\hline
\hline
 $n=1$ & $2$ & & & & & \\[1pt]
\hline
 $2$ & $2s$ & $2$ & & & & \\[1pt]
\hline
 $3$ & $s^2$ & $2s$ & $1$ & & & \\[1pt]
\hline
 $4$ & $-\frac{1}{12}s+\frac{1}{3}s^3$ & $-\frac{1}{12}+s^2$ & $s$ & $\frac{1}{3}$ & & \\[1pt]
\hline
 $5$ & $-\frac{1}{12}s^2+\frac{1}{12}s^4$ & $-\frac{1}{6}s+\frac{1}{3}s^3$
 &$-\frac{1}{12}+\frac{1}{2}s^2$ & $\frac{1}{3}s$ & $\frac{1}{12}$ & \\[1pt]
\hline
 $6$ & $\frac{3}{320}s-\frac{1}{24}s^3+\frac{1}{60}s^5$ & $\frac{3}{320}-\frac{1}{8}s^2+\frac{1}{12}s^4$
 &$-\frac{1}{8}s+\frac{1}{6}s^3$ & $-\frac{1}{24}+\frac{1}{6}s^2$ & $\frac{1}{12}s$ & $\frac{1}{60}$
\end{tabular}
}
}
\end{center}
\caption{The polynomial $T_{n,d}(s)$ for $n=1,2,3,4,5,6$.}
\end{table}

\subsubsection{$f_{n,r}(s)$}

 Let us recall the definition of the polynomial $f_{n,r}(s)$:
\[
 f_{n,r}(s)
:=\sum^{n-1}_{d=1}\frac{(-1)^{r+d}(r-1)!}{2(r+d)_r}H(r,r+d-1)
(2s)^{2r+d-1}T_{n,d}(s).
\]
 Currently, we have not yet obtained any ``closed'' (or ``simplified'') expression of
 $f_{n,r}(s)$, however, one can see the following properties.

\begin{lem}
\label{lem:f}
 $f_{n,r}(s)=f_{n,r}(-s)=(-1)^{n}f_{n,r}(s)$. 
 In particular, $f_{n,r}(s)$ is identically $0$ if $n$ is odd and an
 even polynomial otherwise.        
\end{lem} 
\begin{proof}
 The first equation can be seen from the expression \eqref{for:DetrB}
 together with the facts that both $L_n(s)$ and $\b_{n,r}(s,l)$ are even and 
 $\xi^{\pm}_{n}(-s)=\xi^{\mp}_{n}(s)$.
 The second fact follows from the definition of $f_{n,r}(s)$
 and identity $(\mathrm{i})$ in Lemma~\ref{lem:prop-Tnd}. 
\end{proof}

\begin{example}
 We here give the explicit expression of the polynomial $f_{n,r}(s)$ for $n=2,4,6$
 (we again note that $f_{n,r}(s)=0$ if $n$ is odd).
\begin{align*}
 f_{2,r}(s)
&=\frac{(-1)^{r+1}}{r^2}\binom{2r}{r}^{-1}(2s)^{2r},\\
 f_{4,r}(s)
&=\frac{(-1)^{r+1}}{r^2}\binom{2r}{r}^{-1}(2s)^{2r}
\biggl(-\frac{1}{24}+\frac{3r+1}{12(2r+1)(2r+2)}(2s)^2\biggr),\\
 f_{6,r}(s)
&=\frac{(-1)^{r+1}}{r^2}\binom{2r}{r}^{-1}(2s)^{2r}\\
&\ \ \ \times
\bigg(\frac{3}{640}-\frac{3r+1}{96(2r+1)(2r+2)}(2s)^2+\frac{15r^2+25r+6}{480(2r+1)(2r+2)(2r+3)(2r+4)}(2s)^4\biggr).
\end{align*} 
\end{example}

\begin{example}
 Even if $r=1$, $f_{n,r}(s)$ is complicated.
 For example,  
\begin{align*}
 f_{2,1}(s)
&=2s^2,\\
 f_{4,1}(s)
&=\frac{1}{36}s^2(8s^2-3),\\
 f_{6,1}(s)
&=\frac{1}{43200}s^2(368s^4-1200s^2+405)),\\
 f_{8,1}(s)
&=\frac{1}{50803200}s^2(8448s^6-90160s^4+217560s^2-70875),\\
 f_{10,1}(s)
&=\frac{1}{73156608000}s^2(144128 s^8-3548160s^6+25425120s^4-54247200s^2+17364375).
\end{align*} 
\end{example}

\subsubsection{$\a_{n,r}(s,k)$}

 We give a proof of Lemma~\ref{lem:apm}.

\begin{proof}
[Proof of Lemma~\ref{lem:apm}] 
 It is sufficient to show the equation \eqref{for:gene-alpha}.
 Changing the order of the summation and using the equation \eqref{for:P},
 we have
\begin{align*}
 A^{+}_{n,r}(s,X)
&=-X^{r-1}\sum^{n-1}_{d=0}
\Biggl(\sum^{r+d-1}_{k=0}\binom{r+d-1}{k}X^{k}(-2s)^{r+d-1-k}\Biggr)T_{n,d}(s)\\
&=-X^{r-1}(X-2s)^{r-1}\sum^{n-1}_{d=0}T_{n,d}(s)\bigl((X-s)-s\bigr)^d\nonumber\\
&=-P_n(X-s)X^{r-1}(X-2s)^{r-1}.\nonumber
\end{align*}
 Similarly one has 
\begin{align*}
 A^{-}_{n,r}(-s,X)
&=-X^{r-1}\sum^{n-1}_{d=0}
\Biggl(\sum^{r-1}_{k=0}\binom{r-1}{k}(-2s)^{k}X^{r-1-k}\Biggr)T_{n,d}(-s)X^d\\
&=-X^{r-1}(X-2s)^{r-1}\sum^{n-1}_{d=0}T_{n,d}(-s)\bigl((X-s)+s\bigr)^d\nonumber\\
&=-P_n(X-s)X^{r-1}(X-2s)^{r-1}.\nonumber
\end{align*}
 This shows the claim.
\end{proof}

 The following expression is more convenient than
 \eqref{for:ap} or \eqref{for:am}.

\begin{prop}

\begin{equation}
\label{for:ar}
 \a_{n,r}(s,k)
=(-1)^{n+k+1}\sum^{\min\{r-1,k-r\}}_{j=0}\binom{r-1}{j}(2s)^{r-1-j}T_{n,k-r-j}(s). 
\end{equation}
\end{prop}
\begin{proof}
 From expressions \eqref{for:gene-alpha} and \eqref{for:P}, we have
\begin{align*}
 A_{n,r}(s,X)
&=-X^{r-1}\sum^{n-1}_{d=0}T_{n,d}(s)(X-2s)^{d+r-1}\\
&=-\sum^{n-1}_{d=0}T_{n,d}(s)\sum^{d+r-1}_{l=0}\binom{d+r-1}{l}(-2s)^{d+r-1-l}X^{r+l-1}\\
&=-\sum^{n+2r-2}_{k=r}\Biggl(\sum^{n-1}_{d=\max\{k-2r+1,0\}}\binom{d+r-1}{k-r}(-2s)^{d+2r-1-k}T_{n,d}(s)\Biggr)X^{k-1}.
\end{align*}  
 Hence, comparing the coefficient of $X^{k-1}$, we have 
\[
 \a_{n,r}(s,k)
=-\sum^{n-1}_{l=\max\{k-2r+1,0\}}\binom{l+r-1}{k-r}(-2s)^{l+2r-1-k}T_{n,l}(s).
\]
 Here, using the equation $\binom{n}{k}=\sum^{l}_{j=0}\binom{l}{j}\binom{n-l}{k-j}$,
 we have
\begin{align*}
 \a_{n,r}(s,k)
&=-\sum^{r-1}_{j=0}\binom{r-1}{j}(-2s)^{r-1-j}\sum^{n-1}_{l=\max\{k-2r+1,0\}}\binom{l}{k-r-j}(-2s)^{l-(k-r-j)}T_{n,l}(s)\\
&=-\sum^{\min\{r-1,k-r\}}_{j=0}\binom{r-1}{j}(-2s)^{r-1-j}\sum^{n-1}_{l=k-r-j}\binom{l}{k-r-j}(-2s)^{l-(k-r-j)}T_{n,l}(s). 
\end{align*}
 Therefore, the desired claim follows from
 $(\mathrm{i})$ and $(\mathrm{iii})$ in Lemma~\ref{lem:prop-Tnd}.
\end{proof}

\begin{example}
 From the equation \eqref{for:ar}, we have   
\begin{align*}
 \a_{n,1}(s,k)
&=(-1)^{n+k+1}T_{n,k-1}(s),\\
 \a_{n,2}(s,k)
&=(-1)^{n+k+1}\Bigl(2sT_{n,k-2}(s)+T_{n,k-3}(s)\Bigr),\\
 \a_{n,3}(s,k)
&=(-1)^{n+k+1}\Bigl((2s)^2T_{n,k-3}(s)+2(2s)T_{n,k-4}(s)+T_{n,k-5}(s)\Bigr).
\end{align*}
\end{example}

\subsubsection{$\b_{n,r}(s,l)$}

 Similar to the previous section, we first give a proof of Lemma~\ref{lem:bpm}.

\begin{proof}
[Proof of Lemma~\ref{lem:bpm}]
 In this case, it is enough to show the equation \eqref{for:gene-b}.
 Changing the order of the summations and using the generating function \eqref{for:genec},
 we have
\begin{align*}
 B^{\pm}_{n,r}(s,Y)
&=\sum^{2r+n-2}_{k=r}\a_{n,r}(\pm s,k)\bigl(Y+\xi^{\pm}_{n}(s)\bigr)^{k-1}
=A_{n,r}\bigl(\pm s,Y+\xi^{\pm}_{n}(s)\bigr).
\end{align*}
 Hence, the expression \eqref{for:gene-b} follows from the formula \eqref{for:gene-alpha}.
\end{proof}

 From the generating function \eqref{for:gene-b},
 one obtains the following expression of $\b_{n,r}(s,l)$.

\begin{prop}

\begin{align}
\label{for:br}
 \b_{n,r}(s,l)=
\begin{cases}
 -(\check{n}^2-s^2)^{r-1} & (1\le l\le n-1), \\[3pt]
 \DS{-\sum^{2r-2}_{k=l-n}(-1)^kc^{k}_{n}(l-n)}\\
\DS{ \ \times\sum_{0\le p,q\le r-1\atop p+q=k}
\binom{r-1}{p}\binom{r-1}{q}\bigl(\xi^{+}_{n}(s)\bigr)^{r-1-p}\bigl(\xi^{-}_{n}(s)\bigr)^{r-1-q}} &
 (n\le l\le n+2r-2).
\end{cases}
\end{align}
 Here, for $0\le l\le k$,
 $c^{k}_{n}(l)$ is defined by the difference equation 
\begin{equation}
\label{def:c}
 c^{k}_{n}(l)=
\begin{cases}
 nc^{k-1}_{n}(0)-(n-1) & (l=0),\\
 (n+l)c^{k-1}_n(l)-(n+l-1)c^{k-1}_{n}(l-1) & (1\le l\le k-1), \\
 -(n+k-1)c^{k-1}_{n}(k-1) & (l=k)
\end{cases}
\end{equation}
 with the initial condition $c^{0}_{n}(0):=2$.
 In particular, 
\begin{align}
\label{for:bn}
 \b_{n,r}(s,n)
&=-2(\check{n}^2-s^2)^{r-1}, \\
\label{for:bn+2r-2}
  \b_{n,r}(s,n+2r-2)
&=-2(n)_{2r-2}.
\end{align}
\end{prop}
\begin{proof}
 Let $T=\frac{t-1}{t}$. 
 Let us calculate the sum $\sum^{\infty}_{Y=0}B_{n,r}(s,Y)T^Y$ in two ways.
 First, from the definition, using the binomial theorem, we have
\begin{align}
\label{for:B1}
 \sum^{\infty}_{Y=0}B_{n,r}(s,Y)T^Y
&=\sum^{2r+n-2}_{l=1}\b_{n,r}(s,l)(1-T)^{-l} 
=\sum^{2r+n-2}_{l=1}\b_{n,r}(s,l)t^l.
\end{align}
 On the other hand, 
 from the equation \eqref{for:gene-b}, we have  
\begin{align*}
 \sum^{\infty}_{Y=0}B_{n,r}(s,Y)T^Y
&=-\sum^{r-1}_{p=0}\sum^{r-1}_{q=0}\binom{r-1}{p}\binom{r-1}{q}\xi^{p,q}_{n,r}(s)
\sum^{\infty}_{Y=0}P_n(Y+1+\check{n})Y^{p+q}T^{Y}\nonumber \\
&=-\sum^{r-1}_{p=0}\sum^{r-1}_{q=0}\binom{r-1}{p}\binom{r-1}{q}\xi^{p,q}_{n,r}(s)
\Bigl(T\frac{d}{dT}\Bigr)^{p+q}\sum^{\infty}_{Y=0}P_n(Y+1+\check{n})T^{Y},\nonumber 
\end{align*}
 where, for simplicity,
 we put $\xi^{p,q}_{n,r}(s):=(\xi^{+}_{n}(s))^{r-1-p}(\xi^{-}_{n}(s))^{r-1-q}$.
 Here, since $1-T=t^{-1}$, 
 the inner sum can be calculated as 
\begin{align*}
 \sum^{\infty}_{Y=0}P_n(Y+1+\check{n})T^{Y}
&=\frac{1}{T}\bigl((1-T)^{-(n+1)}-1\bigr)-T(1-T)^{-(n+1)}\\
&=\frac{1}{1-(1-T)}\Bigl(2(1-T)^{-n}-(1-T)^{-n+1}-1\Bigr)\\
&=2\sum^{\infty}_{l=0}t^{n-l}-\sum^{\infty}_{l=0}t^{n-1-l}-\sum^{\infty}_{l=0}t^{-l}\\
&=\sum^{n-1}_{l=1}t^{l}+2t^{n}.
\end{align*}
 Hence, from the equation $T\frac{d}{dT}=t(t-1)\frac{d}{dt}$, we have
\begin{align*}
 \Bigl(T\frac{d}{dT}\Bigr)^{k}
\sum^{\infty}_{Y=0}P_n(Y+1+\check{n})T^{Y}
&=\Bigl(t(t-1)\frac{d}{dt}\Bigr)^k\Biggl(\sum^{n-1}_{l=1}t^{l}+2t^{n}\Biggl)
=(-1)^k\Biggl(\sum^{n-1}_{l=1}t^l+\sum^{k}_{l=0}c_n^{k}(l)t^{n+l}\Biggr).
\end{align*}
 Substituting this into the above expression with $k=p+q$, we have 
\begin{align}
\label{for:B2}
 \sum^{\infty}_{Y=0}B_{n,r}(s,Y)T^Y
&=-\sum^{r-1}_{p=0}\sum^{r-1}_{q=0}\binom{r-1}{p}\binom{r-1}{q}
 \xi^{p,q}_{n,r}(s)(-1)^{p+q}\sum^{n-1}_{l=1}t^l \\
&\ \ \ -\sum^{r-1}_{p=0}\sum^{r-1}_{q=0}\binom{r-1}{p}\binom{r-1}{q}
 \xi^{p,q}_{n,r}(s)(-1)^{p+q}\sum^{p+q}_{l=0}c_n^{p+q}(l)t^{n+l} \nonumber\\
&=-\sum^{n-1}_{l=1}(\check{n}^2-s^2)^{r-1}t^l \nonumber\\  
&\ \ \ -\sum^{n+2r-2}_{l=n}\Biggl(\sum^{2r-2}_{k=l-n}(-1)^kc^{k}_{n}(l-n)
\sum_{0\le p,q\le r-1 \atop
 p+q=k}\binom{r-1}{p}\binom{r-1}{q}\xi^{p,q}_{n,r}(s)\Biggr)t^{l}.
 \nonumber 
\end{align}
 Now, comparing the coefficient of $t^{l}$ in \eqref{for:B1} and \eqref{for:B2},
 one obtains the desired expression.
 In particular, the formulas \eqref{for:bn} and \eqref{for:bn+2r-2} are obtained
 from \eqref{for:br} together with facts $c^{k}_{n}(0)=1+n^{k}$ and
 $c^{k}_{n}(k)=(-1)^k2(n)_k$, respectively.
 This completes the proof. 
\end{proof}

\begin{example}
 From the equation \eqref{for:br}, 
 we have 
\begin{align}
\label{for:b1}
 \b_{n,1}(s,l)&=
\begin{cases}
 -1 & (1\le l\le n-1), \\
 -2 & (l=n);
\end{cases}\\[3pt]
\label{for:b2}
 \b_{n,2}(s,l)&=
\begin{cases}
 -(\check{n}^2-s^2) & (1\le l\le n-1), \\
 -2(\check{n}^2-s^2) & (l=n), \\
 n(n+1) & (l=n+1), \\
 -2n(n+1) & (l=n+2);
\end{cases}\\[3pt]
\label{for:b3}
 \b_{n,3}(s,l)&=
\begin{cases}
 -(\check{n}^2-s^2)^2 & (1\le l\le n-1), \\
 -2(\check{n}^2-s^2)^2 & (l=n), \\
 -2n(n+1)s^2+\frac{1}{2}n(n+1)(n^2+1) & (l=n+1), \\
 4n(n+1)s^2-n(n+1)(2n^2+5n+7) & (l=n+2), \\
 3n(n+1)(n+2)(n+3) & (l=n+3), \\
 -2n(n+1)(n+2)(n+3) & (l=n+4). \\
\end{cases}
\end{align}
\end{example}

\begin{remark}
 Note that $c^{k}_n(l)$ has the following closed expression:   
\begin{equation}
\label{for:closedc}
 c^k_{n}(l)=1+\binom{n+l-1}{l}\sum^{l}_{j=0}(-1)^j\binom{l}{j}\frac{n+2j-1}{n+j-1}(n+j)^k.
\end{equation}
 In fact, it is easy to check that the righthand-side of \eqref{for:closedc}
 satisfies the same difference equation \eqref{def:c} with the initial condition.
 (To see this, additionally, one has to show that $c^k_{n}(l)=0$ for $l>k$.)
\end{remark}

\section{The case $r=1$}
\label{sec:r=1}

 In this section, as corollaries of Theorem~\ref{thm:main}, 
 we give explicit expressions of 
 the usual determinant of the Laplacian $\Delta_n$
 and conformal Laplacian $Y_n$ in terms of the derivatives of the Riemann zeta function.

\begin{cor}
\label{cor:r=1}
 $\mathrm{(i)}$ It holds that  
\begin{align}
\label{for:detDn}
 \det\bigl(\Delta_n\bigr)
%&=\regprod^{\infty}_{k=1}\l_{n,k}
&=e^{f_{n,1}(\che{n})}
\prod^{n-1}_{l=1}
\bigl(\G_l(n)\G_l(1)\bigr)^{-1}\cdot \bigl(\G_n(n)\G_n(1)\bigr)^{-2}\\
\label{for:detDnexplicit}
&=
\begin{cases}
 \DS{4\pi^2} & (n=1), \\
 \DS{\frac{1}{n-1}e^{f_{n,1}(\check{n})+\sum^{n-1}_{k=0}d_n(k)\z'(-k)}} & (n\ge 2).
\end{cases}
\end{align}
 Here, for $0\le k\le n-1$, $d_n(k)$ is the rational number defined by the generating function
\begin{equation}
\label{def:genedkn}
 \sum^{n-1}_{k=0}d_n(k)z^k
=-\binom{z+n-2}{n-1}\frac{n+2z-1}{z}+(-1)^n\binom{-z+n-2}{n-1}\frac{n-2z-1}{-z}.
\end{equation}
 In particular, $d_n(k)=0$ if $k$ and $n$ have the same parity.
 
 $\mathrm{(ii)}$ It holds that 
\begin{align}
\label{for:detYn}
 \det\bigl(Y_n\bigr)
%&=\regprod^{\infty}_{k=1}\l_{n,k}\bigl(\frac{1}{2}\bigr)
&=\bigl(\frac{n(n-2)}{4}\bigr)^{\d_{n\ge 3}}\cdot e^{f_{n,1}(\frac{1}{2})}
\prod^{n-1}_{l=1}
\Bigl(\G_l\bigl(\frac{n}{2}+1\bigr)\G_l\bigl(\frac{n}{2}\bigr)\Bigr)^{-1}\cdot 
\Bigl(\G_n\bigl(\frac{n}{2}+1\bigr)\G_n\bigl(\frac{n}{2}\bigr)\Bigr)^{-2}\\
\label{for:detYnexplicit}
&=
\begin{cases}
 \DS{16}
 & (n=1),\\
 \DS{e^{\frac{1}{2}-4\z'(-1)}}
 & (n=2), \\
 \DS{2^{\frac{(-1)^{m+1}(2m-3)!!}{2^{2(m-1)}(2m)!!}}
e^{\sum^{2m}_{k=0 }(2^{-k}-1)y_{2m+1}(k)\z'(-k)}}
 & (n=2m+1, \ m\ge 1),\\
 \DS{
e^{f_{2m,1}(\frac{1}{2})
+\sum^{2m-1}_{k=0}y_{2m}(k)\z'(-k)}}
 & (n=2m, \ m\ge 2).
\end{cases}
\end{align}
 Here, for $0\le k\le n-1$, $y_n(k)$ is the rational number defined by the generating function
\begin{align}
\label{def:geneyknodd}
 \sum^{2m}_{k=0}y_{2m+1}(k)z^k
&=(-1)^{m}\binom{2m-2}{m-1}^{-1}\binom{z+m-\frac{3}{2}}{m-1}\binom{-z+m-\frac{3}{2}}{m-1}\\
&\ \ \
 +4(-1)^{m+1}\binom{2m}{m}^{-1}\binom{z+m-\frac{1}{2}}{m-1}\binom{-z+m-\frac{1}{2}}{m-1}\nonumber\\
\intertext{if $n=2m+1$, for $m\ge 1$ and}
\label{def:geneykneven}
\sum^{2m-1}_{k=0}y_{2m}(k)z^k
&=\frac{(-1)^{m}(2m-2)}{z}\binom{2m-2}{m-1}^{-1}\binom{z+m-2}{m-1}\binom{-z+m-2}{m-1}\\
&\ \ \ +4\frac{(-1)^{m+1}2m}{z}\binom{2m}{m}^{-1}\binom{z+m-1}{m-1}\binom{-z+m-1}{m-1}\nonumber
\end{align}
 if $n=2m$, for $m\ge 2$. 
 In particular, $y_n(k)=0$ if $k$ and $n$ have the same parity.
\end{cor}

\begin{remark}
 As we have seen in \eqref{for:vardi},
 it is well known that
 (the logarithm of) the determinant $\det(\Delta_n)$
 can be essentially written as
 a linear combination of the derivatives of the Riemann zeta function.
 However, there are few papers in which the coefficients are explicitly obtained.
 In this sense, the expression \eqref{for:detDnexplicit} is meaningful
 because one can explicitly calculate the coefficients $d_{n}(k)$ from \eqref{def:genedkn}. 
\end{remark}

\begin{proof}
[Proof of Corollary~\ref{cor:r=1}]
 The expressions \eqref{for:detDn} and \eqref{for:detYn} follow immediately
 from \eqref{for:DetrDn} and \eqref{for:DetrYn}, respectively, 
 together with the formula \eqref{for:b1}.
 We first derive the expression \eqref{for:detDnexplicit}.
 The case $n=1$ is clear: 
\[
 \det(\Delta_1)=\G_1(1)^{-4}=\bigl(\sqrt{2\pi}\bigr)^4=4\pi^2.
\] 
 Now suppose $n\ge 2$.
 Then, using the ladder relation \eqref{for:ladder} and the formula \eqref{for:ladderformula},
 one can see that 
\[
 \det\bigl(\Delta_n\bigr)
=\frac{1}{n-1}e^{f_{n,1}(\che{n})}\prod^{n}_{l=1}\G_l(1)^{p_n(l)},
\]
 where, for $1\le l\le n$, $p_{n}(l)$ is defined by 
\begin{equation*}
%\label{def:pnl}
 p_{n}(l):=
\begin{cases}
 \DS{-1-(-1)^{n-l}\Biggl(2\binom{n-2}{l-2}+\binom{n-2}{l-1}\Biggr)} & (1\le l\le n-1),\\
 -4 & (l=n).
\end{cases}
\end{equation*}
 Hence, from \eqref{for:G1}, changing the order of the products, we have
\[
 \det\bigl(\Delta_n\bigr)
=\frac{1}{n-1}e^{f_{n,1}(\che{n})+\sum^{n-1}_{k=0}d_n(k)\z'(-k)},
\]
 where $d_n(k):=\sum^{n}_{l=k+1}b_{l,k}(1)p_n(l)$.
 Put $\cD_n(z):=\sum^{n}_{k=0}d_n(k)z^k$.
 It is enough to show that
 $\cD_n(z)$ coincides with the righthand-side of \eqref{def:genedkn}
 (if true, the assertion $d_n(k)=0$ if $k\equiv n$ $(\mathrm{mod}\,2)$
 automatically follows from the fact $\cD_n(z)=(-1)^{n-1}\cD_n(-z)$).
 From \eqref{def:b}, we have $\sum^{l-1}_{k=0}b_{l,k}(1)z^k=\binom{z+l-2}{l-1}$.
 Hence, changing the order of the summations, 
 we have $\cD_n(z)=\sum^{n}_{l=1}\binom{z+l-2}{l-1}p_n(l)$.
 Therefore, one obtains
\begin{align*}
 \cD_n(z)
&=-\sum^{n-1}_{l=1}\binom{z+l-2}{l-1}-2(-1)^n\sum^{n-1}_{l=1}\binom{z+l-2}{l-1}\binom{n-2}{l-2}(-1)^l\\
&\ \ \ 
-(-1)^n\sum^{n-1}_{l=1}\binom{z+l-2}{l-1}\binom{n-2}{l-1}(-1)^l-4\binom{z+n-2}{n-1}\\
&=-\frac{n-1}{z}\binom{z+n-2}{n-1}-2(-1)^n
\Biggl(-\binom{-z+n-2}{n-1}-(-1)^n\binom{z+n-2}{n-1}\Biggr)\\
&\ \ \ +(-1)^n\frac{n-1}{-z}\binom{-z+n-2}{n-1}-4\binom{z+n-2}{n-1}\\
&=-\binom{z+n-2}{n-1}\frac{n+2z-1}{z}+(-1)^n\binom{-z+n-2}{n-1}\frac{n-2z-1}{-z}.
\end{align*}
 This shows the desired claim.

 The expression \eqref{for:detYnexplicit} is similarly obtained.
 In fact, the case $n=1$ and $n=2$ are respectively obtained as
\begin{align*}
 \det(Y_1)
&=\Bigl(\frac{1}{4}\G_1\bigl(\frac{1}{2}\bigr)^4\Bigr)^{-1}
=\Bigl(\frac{1}{4}\bigl(\frac{1}{\sqrt{2}}\bigr)^4\Bigr)^{-1}
=16,\\
 \det(Y_2)
&=e^{f_{2,1}(\frac{1}{2})}\bigl(\G_2(1)\bigr)^{-4}
=e^{\frac{1}{2}}(e^{\z'(-1)})^{-4}
=e^{\frac{1}{2}-4\z'(-1)}.
\end{align*} 
 We notice that, from the definition, $\det(Y_2)=\det(\Delta_2)$.
 Now assume $n\ge 3$.
 Let us write $n=2m+1$ if $n$ is odd and $n=2m$ otherwise.
 Using \eqref{for:ladder} and \eqref{for:ladderformula} again,
 we have 
\begin{align*}
 \det\bigl(Y_{2m+1}\bigr)
&=\prod^{2m+1}_{l=1}\G_l\bigl(\frac{1}{2}\bigr)^{q_{2m+1}(l)},\\
 \det\bigl(Y_{2m}\bigr)
&=e^{f_{2m,1}(\frac{1}{2})}\prod^{2m}_{l=1}\G_l(1)^{q_{2m}(l)},
\end{align*}
 where, for $1\le l\le n$, $q_{n}(l)$ is defined by 
\begin{align*}
 q_{2m+1}(l)
&:=(-1)^l\Biggl(\binom{m-1}{l-m}+4\binom{m}{l-m-1}\Biggr),\\
 q_{2m}(l)
&:=(-1)^{l+1}\Biggl(\binom{m-2}{l-m}+4\binom{m-1}{l-m-1}\Biggr).
\end{align*}
 We notice that, from Lemma~\ref{lem:f}, $f_{2m+1,1}(\frac{1}{2})=0$.
 Moreover, from \eqref{for:G1/2} and \eqref{for:G1}, 
 changing the order of the products, 
 we have 
\begin{align*}
 \det\bigl(Y_{2m+1}\bigr)
&=2^{-\sum^{2m}_{k=0}\frac{2^{-k}B_{k+1}}{k+1}y_{2m+1}(k)}
 e^{\sum^{2m}_{k=0}(2^{-k}-1)y_{2m+1}(k)\z'(-k)},\\
 \det\bigl(Y_{2m}\bigr)
&=e^{f_{2m,1}(\frac{1}{2})+\sum^{2m-1}_{k=0}y_{2m}(k)\z'(-k)},
\end{align*} 
 where,
 for $0\le k\le 2m$,
 $y_{2m+1}(k):=\sum^{2m+1}_{l=k+1}b_{l,k}(\frac{1}{2})q_{2m+1}(l)$ and,
 for $0\le k\le 2m-1$,
 $y_{2m}(k):=\sum^{2m}_{l=k+1}b_{l,k}(1)q_{2m}(l)$, respectively.
 Let $\cY_{n}(z):=\sum^{n-1}_{k=0}y_{n}(k)z^k$.
 We claim that $\cY_{2m+1}(z)$ (resp. $\cY_{2m}(z)$) is equal to 
 the righthand-side of \eqref{def:geneyknodd} (resp. \eqref{def:geneykneven}). 
 We here prove only the case $n=2m+1$ (the other case is similar).
 From \eqref{def:b}, it holds that  
 $\sum^{l-1}_{k=0}b_{l,k}(\frac{1}{2})z^k=\binom{z+l-\frac{3}{2}}{l-1}$.
 Therefore, changing the order of the summations, 
 we have 
\begin{align*}
 \cY_{2m+1}(z)
&=\sum^{2m-1}_{l=m}\binom{z+l-\frac{3}{2}}{l-1}\binom{m-1}{l-m}(-1)^l
+4\sum^{2m+1}_{l=m+1}\binom{z+l-\frac{3}{2}}{l-1}\binom{m}{l-m-1}(-1)^l\\
&=(-1)^{m}\binom{2m-2}{m-1}^{-1}\binom{z+m-\frac{3}{2}}{m-1}\binom{-z+m-\frac{3}{2}}{m-1}\\
&\ \ \ +4(-1)^{m+1}\binom{2m}{m}^{-1}\binom{z+m-\frac{1}{2}}{m}\binom{-z+m-\frac{1}{2}}{m},
\end{align*}
 whence one obtains the claim.
 Finally, since $y_{2m+1}(z)=0$ if $k$ is odd and
 $B_{k+1}=0$ if $k$ is even for $k\ge 2$, 
 we have 
\begin{align*}
 -\sum^{2m}_{k=0}\frac{2^{-k}B_{k+1}}{k+1}y_{2m+1}(k)
=-B_1\cdot y_{2m+1}(0)=-\frac{1}{2}\cdot \cY_{2m+1}(0)
%&=\frac{1}{2}\Biggl[(-1)^m\binom{2m-2}{m-1}^{-1}\binom{m-\frac{3}{2}}{m-1}^{2}
%+4(-1)^{m+1}\binom{2m}{m}^{-1}\binom{m-\frac{1}{2}}{m}^{2}\Biggr]
=\frac{(-1)^{m+1}(2m-3)!!}{2^{2(m-1)}(2m)!!}.
\end{align*}
 This completes the proofs.
\end{proof}

\begin{example}
 From \eqref{for:detDnexplicit}, we have 
\begin{align*}
 \det\bigl(\Delta_1\bigr)
&=4\pi^2
=39.47841760\ldots,\\
 \det\bigl(\Delta_2\bigr)
&=e^{\frac{1}{2}-4\z'(-1)}
=3.195311486\ldots,\\
 \det\bigl(\Delta_3\bigr)
&=\frac{1}{2}e^{-2\z'(0)-2\z'(-2)}
=3.338851214\ldots,\\
 \det\bigl(\Delta_4\bigr)
&=\frac{1}{3}e^{\frac{15}{16}-\frac{13}{3}\z'(-1)-\frac{2}{3}\z'(-3)}
=1.736943483\ldots,\\
 \det\bigl(\Delta_5\bigr)
&=\frac{1}{4}e^{-2\z'(0)-\frac{23}{6}\z'(-2)-\frac{1}{6}\z'(0)}
=1.762919348\ldots,\\
 \det\bigl(\Delta_6\bigr)
&=\frac{1}{5}e^{\frac{455}{432}-\frac{149}{30}\z'(-1)-2\z'(-3)-\frac{1}{30}\z'(-5)}
=1.290018366\ldots.
\end{align*}
 We notice that it is shown in \cite{Moller2009} that
 $\lim_{m\to\infty}\det(\Delta_{2m+1})=0$. 
\end{example}

\begin{example}
 From \eqref{for:detYnexplicit}, we have 
\begin{align*}
 \det\bigl(Y_1\bigr)
&=16,\\
 \det\bigl(Y_2\bigr)
&=e^{\frac{1}{2}-4\z'(-1)}
=3.195311486\ldots,\\
 \det\bigl(Y_3\bigr)
&=2^{\frac{1}{4}}e^{-2\z'(-2)}
=1.136114502\ldots,\\
 \det\bigl(Y_4\bigr)
&=e^{-\frac{1}{144}-\frac{1}{3}\z'(-1)-\frac{2}{3}\z'(-3)}
=1.045620218\ldots,\\
 \det\bigl(Y_5\bigr)
&=2^{-\frac{1}{64}}e^{\frac{1}{16}\z'(-2)+\frac{5}{32}\z'(-4)}
=0.9885797293\ldots,\\
 \det\bigl(Y_6\bigr)
&=e^{\frac{1}{1350}+\frac{1}{30}\z'(-1)-\frac{1}{30}\z'(-5)}
=0.9952570855\ldots.
\end{align*}
 We also notice that it is shown in \cite{Moller2009} that
 $\lim_{n\to\infty}\det(Y_n)=1$. 
\end{example}

%\section{Dimensional asymptotics of the higher depth determinant}

%=========================================================================
%	References
%=========================================================================

%%%%%%%%%%%%%%%%%%%%%%%%%%%%%%%%%%%%%%%%%%%%%%%%%%%%%%%%%%%%%%%%%%%%%%%%%% 

\bigskip

\noindent
\textsc{Yoshinori YAMASAKI}\\
 Graduate School of Science and Engineering, Ehime University.\\
 Bunkyo-cho, Matsuyama 790-8577 JAPAN.\\
 \texttt{yamasaki@math.sci.ehime-u.ac.jp}\\
 Tel\,:\,+81-(0)89-927-9554 \\
 Fax\,:\,+81-(0)89-927-9560

%===========================================================================
%===========================================================================
\end{document}